\documentclass[12pt]{amsart}
\usepackage[margin=1in]{geometry}

\usepackage{url,amsmath,amssymb,mathrsfs}
\usepackage{enumitem}
\usepackage{commath}
\usepackage{physics}
\usepackage{nicefrac}
\usepackage{amsthm}
\usepackage{xfrac}
\usepackage[all]{xy}
\usepackage{graphicx}
\graphicspath{ {./images/} }
\usepackage{color}
\usepackage{csquotes}

\newtheorem{theorem}{Theorem}[section]
\newtheorem{lemma}[theorem]{Lemma}
\newtheorem{proposition}[theorem]{Proposition}
\newtheorem{corollary}[theorem]{Corollary}

\newtheorem{conjecture}[theorem]{Conjecture}

\newtheorem{innercustomgeneric}{\customgenericname}
\providecommand{\customgenericname}{}
\newcommand{\newcustomtheorem}[2]{%
	\newenvironment{#1}[1]
	{%
		\renewcommand\customgenericname{#2}%
		\renewcommand\theinnercustomgeneric{##1}%
		\innercustomgeneric
	}
	{\endinnercustomgeneric}
}
\newcustomtheorem{customthm}{Theorem}
\newcustomtheorem{customlemma}{Lemma}

\theoremstyle{definition}
\newtheorem{example}[theorem]{Example}
\newtheorem{definition}[theorem]{Definition}

\newtheorem{remark}[theorem]{Remark}

\newcommand{\bZ}{\mathbb Z}
\newcommand{\bQ}{\mathbb Q}
\newcommand{\bN}{\mathbb N}

\newcommand{\Rbar}{\overline{R}}

\newcommand{\Cl}{\textup{Cl}}

\newcommand{\aut}{\textup{Aut}}

\newcommand{\modulo}[1]{\textup{ (mod }#1)}

\newcommand{\gen}[1]{\langle#1\rangle}

\newcommand{\ow}{\textup{otherwise}}

\newcommand{\cO}{\mathcal{O}}
\newcommand{\Irr}{\textup{Irr}}
\newcommand{\term}[1]{\textbf{\textup{#1}}}

\newcommand{\quot}[2]{\large\sfrac{#1}{#2}\normalsize}

\title[Elasticity of Orders from the $S$-Relative Davenport Constant]{Elasticity of Orders from the $S$-relative Davenport Constant: an Arithmetic Application of a Number-Theoretic Investigation}
\author{Jared Kettinger \and Grant Moles}

\begin{document}

\begin{abstract}
    Orders in algebraic number fields have long been objects of central interest in algebraic number theory. Despite non-maximal orders failing to be Dedekind, the present authors have previously shown that the structure of the ideal class group may still contain enough information to determine elasticity. In this paper, we develop the $S$-relative Davenport constant, which builds on previous work by M. Ska{\l}ba. Although Ska{\l}ba's original construction was defined to aid in the study of binary quadratic forms, we show that this related invariant is the exact tool needed to tackle the question of elasticity in non-integrally closed orders. In particular, we investigate the elasticity of orders $\cO$ whose conductor ideal $I=(\cO:\cO_K)$ is prime as an ideal of $\cO$, as well as orders in quadratic number fields with primary conductor. We also give conditions under which $\cO$ will have the same elasticity as the full ring of integers $\cO_K$.
\end{abstract}

\maketitle

\section{Introduction}


Let $G$ be a finite abelian group. A $G$\term{-sequence} $\{g_1,g_2,\dots,g_n\}$ of (not necessarily distinct) elements of $G$ is called a \term{zero-sum sequence} if $g_1 + g_2 +\dotsm +g_n = 0$. We use $D(G)$ to denote the smallest positive integer $m$ such that any $G$-sequence of length $m$ must have a zero-sum subsequence. Equivalently, we may define $D(G)$ as the largest positive integer $m$ such that there exists a zero-sum sequence of length $m$ with no proper zero-sum subsequence. This group invariant is known as the \term{Davenport constant} of $G$.

The Davenport constant was first introduced by Rogers (\cite{rogers1963combinatorial}) and later popularized by Davenport (\cite{davenport1966midwestern}) in connection with the arithmetic of rings of integers. In particular, when $G$ is taken to be the class group of a ring of integers $\mathcal{O}_K$, $D(G)$ is a sharp upper bound for the length of the prime factorization (counting multiplicity) of the ideal $(\alpha)$ for an irreducible element $\alpha \in \text{Irr}(\mathcal{O}_K)$. More generally, the Davenport constant is the precise tool we need to take advantage of uniqueness of ideal factorization---allowing us to discern arithmetic information about an element $\alpha$ from the ideal $(\alpha)$. Consider the following definition. 

\begin{definition}
        \label{elasticity}
        Let $R$ be an atomic domain and $\alpha\in R$ a nonzero nonunit. The \term{length set} of $\alpha$ in $R$ is
        $$\ell_R(\alpha):=\{k\in\bN\mid\alpha=\pi_1\cdots\pi_k\textup{ for some }\pi_1,\dots,\pi_k\in\Irr(R)\}.$$
        The \term{elasticity} of $\alpha$ in $R$ is
        $$\rho_R(\alpha):=\sup\left\{\frac{m}{n}\:\middle|\:m,n\in\ell_R(\alpha)\right\}.$$
        The \term{elasticity} of $R$ is
        $$\rho(R):=\sup\{\rho_R(\alpha)|\alpha\in R \text{ a nonzero nonunit}\}.$$
    \end{definition}


    Recall that an integral domain in which every nonzero nonunit factors uniquely into primes is called a \term{unique factorization domain} (UFD). Note that the elasticity of a domain is a measure of the failure of unique factorization---in particular, factorization length. The following is a classic result of Narkiewicz (\cite{narkiewicz}) which demonstrates how the Davenport constant of the class group $\Cl(\mathcal{O}_K)$ of a ring of integers $\mathcal{O}_K$ determines this factorization invariant. 

    \begin{theorem}
        Let $\cO_K$ be a ring of integers which is not a UFD, then
        \[
        \rho(\cO_K) = \frac{D(\Cl(\cO_K))}{2}.
        \]
    \end{theorem}

    For a non-maximal order $\mathcal{O}$ of a number field $K$, the relationship between $D(\Cl(\mathcal{O}))$ and $\rho(\mathcal{O})$ is more tenuous as $\mathcal{O}$ fails to be Dedekind, since $\overline{\mathcal{O}}=\mathcal{O}_K$. As we will see, this results in the generalized class group $\Cl(\mathcal{O})$ containing only information about those ideals relatively prime to the conductor ideal $(\mathcal{O}:\mathcal{O}_K)$. However, recent work by the present authors in \cite{KettingerMoles2025elasticity} has shown that, in the case when $(\mathcal{O}:\mathcal{O}_K)$ is a prime ideal of $\mathcal{O}_K$, the elasticity of the order is still completely determined by the structure of $\Cl(\mathcal{O})$. In this paper, we consider the more general case when $(\mathcal{O}:\mathcal{O}_K)$ is a prime ideal of $\mathcal{O}$. In order to do so, we will introduce a variant of the Davenport constant which not only informs our present inquiry, but also provides a natural language for the results of \cite{KettingerMoles2025elasticity}.

    The Davenport constant has long been studies as a purely combinatorial or number-theoretic object (see \cite{olson1} or \cite{bovey1975conditions}, respectively). It has also inspired numerous variants and generalizations in the area of factorization theory. Interestingly, multiple generalizations have been defined and studied from a strictly number-theoretic point of view which were later found to have profound arithmetic implications. For example, Adhikari and Rath introduced weighted Davenport constants in 2006 (\cite{adhikari2006davenport}), after which they were widely studied. It was not until 2014 that Halter-Koch gave an arithmetic interpretation of these Davenport constants in \cite{Halter-Koch_Arithmetical}. Furthermore, in 2022, Boukheche et al. demonstrated in \cite{boukheche2022monoids} that these invariants completely determine the arithmetic of the normset of a ring of integers $\mathcal{O}_K$ for a Galois number field $K$.  

    Similarly, in 1993, Ska{\l}ba introduced the relative Davenport constant as part of his investigation of numbers with a unique representation by a binary quadratic form (\cite{skalba1993numbers}). It turns out, when we generalize this invariant from elements to sets, we find precisely the tool we need to study the arithmetic of non-maximal orders. 
    
    In Section 2, we will focus on developing this group invariant and proving results which will inform our arithmetic inquiries. In Section 3, we will study orders whose conductor ideal is prime in the order, culminating in a characterization of their elasticities in Theorem \ref{elasticity of principal order}. In Section 4, we briefly apply these results to settle prior conjectures on orders whose elasticity is equal to that of its integral closure. Finally, in Section 5, we restrict our attention to orders in quadratic number fields, allowing for a more complete characterization of elasticity in Theorem \ref{elasticity of Rn}. We conclude with Theorem \ref{final}, which settles a particular case of a well-known conjecture on the elasticity of intermediate orders.

    To remain consistent with previous work on this topic and to emphasize the relationship between a (non-maximal) order and the full ring of integers, we will adopt the notation $R$ to refer to an order in a number field $K$ and $\Rbar$ to refer to its integral closure $\cO_K$.
    

\section{The $S$-Relative Davenport Constant}

A natural generalization of the Davenport constant was first introduced by Ska{\l}ba in \cite{skalba1993numbers} to answer a question on binary quadratic forms. He further explored this invariant in its own right in \cite{skalba1998relative}. Before defining it, we introduce some relevant terms. For a finite abelian group $G$, a $G$-sequence $\{g_1,g_2,\dots,g_n\}$ is called \term{irreducible} if it has no proper zero-sum subsequence. Furthermore, if $\{g_1,g_2,\dots,g_n\}$ also fails to be a zero-sum sequence, we call it \term{primitive}. The \term{relative Davenport constant of $G$ with respect to $a$}, denoted $D_a(G)$, is the largest positive integer $m$ such that there exists an irreducible $G$-sequence $\{g_1,g_2,\dots,g_m\}$ with $g_1 + g_2 + \dotsm +g_m = a$. We will presently develop an analogue of the relative Davenport constant replacing the element $a$ with a subset. Taking advantage of the groundwork done in \cite{skalba1993numbers}, we are able to give explicit formulas for this constant for certain subsets of arbitrary cyclic groups which will translate to strong arithmetic results in later sections. 

\begin{definition}
    Let $G$ be a finite abelian group and $S$ a nonempty subset of $G$. We call the sequence of elements $\{g_1,\dots,g_n\}$ an \term{$S$-sum sequence} in $G$ if $g_1+\dots+g_n\in S$. If $S$ is the singleton set $S=\{g\}$ for some $g\in G$, we may use the term \term{$g$-sum sequence} rather than $\{g\}$-sum sequence for convenience. Note that this agrees with the existing definition of a zero-sum sequence.
\end{definition}

\begin{definition}
    Let $G$ be a finite abelian group and $S$ a nonempty subset of $G$. We define the \term{$S$-relative Davenport constant}, denoted $D_S(G)$, to be the minimal natural number $n$ such that any $S$-sum sequence $\{g_1,\dots,g_n\}$ of length $n$ in $G$ must contain a zero-sum subsequence.
\end{definition}

It should be observed that this definition differs from the generalization introduced in \cite{kettinger2026generalized} which uses similar notation. The following proposition demonstrates the relationship between the S-relative Davenport constant and the original invariant.

\begin{proposition}
    Let $G$ be a finite abelian group. Then $D_G(G)=D(G)$.
\end{proposition}

Note that the $S$-relative Davenport constant and Ska{\l}ba's original construction generalize the two equivalent definitions of $D(G)$---in terms of sequences of minimal and maximal lengths respectively. With this in mind, we introduce a \enquote{small} version of our invariant. 

\begin{definition}\label{small}
    Let $G$ be a finite abelian group and $S$ a nonempty subset of $G$. We define the \term{small $S$-relative Davenport constant}, denoted $d_S(G)$, to be the largest whole number $n$ such that there exists an $S$-sum sequence in $G$ of length $n$ which has no zero-sum subsequence. If every $S$-sum sequence in $G$ has a zero-sum subsequence, we define $d_S(G)=0$.
\end{definition}

The terms small and large Davenport constant are used analogously throughout the literature. Namely, $d(G)$ denotes the largest possible integer $\ell$ such that there exists a $G$-sequence of length $\ell$ with no $0$-subsequence---proper or otherwise. For abelian groups, we have $d(G) + 1 = D(G)$, but this fails in general for non-abelian groups. For details, see \cite{geroldinger2013large}. The following results demonstrate that we avoid this bifurcation for the $S$-relative Davenport constant of abelian groups as well.


\begin{lemma}
    Let $G$ be a finite abelian group and $S$ a nonempty subset of $G$. Then $d_S(G)=0$ if and only if $S=\{0\}$.
\end{lemma}

\begin{proof}
    Recall that $d_S(G)$ is defined to be $0$ exactly when every $S$-sequence in $G$ has a zero-sum subsequence. Assume $d_S(G)=0$, and let $g\in S$. Then the $S$-sequence $(g)$ must contain a zero-sum subsequence, and thus $g=0$. Therefore, $S=\{0\}$.

    Now assume that $S=\{0\}$. Then any $S$-sequence in $G$ must by definition also be a zero-sum sequence and thus contains a zero-sum subsequence. Then $d_S(G)=0$.
\end{proof}

\begin{theorem}
    Let $G$ be a finite abelian group and $S$ a nonempty subset of $G$. Then $D_S(G)=d_S(G)+1$.
\end{theorem}

\begin{proof}
    First, assume that $d_S(G)=0$; that is, assume $S=\{0\}$. Then since any $S$-sequence of length $1$ in $G$ must contain a zero-sum subsequence (namely, the entire sequence), $D_S(G)=1=d_S(G)+1$.

    Now assume that $d_S(G)=d>0$. Then there is some $S$-sequence $(g_1,\dots,g_d)$ in $G$ of length $d$ which has no zero-sum subsequence. For each $1\leq k\leq d$, let $h_k=g_k+g_{k+1}+\dots+g_d$, and consider the sequence $(g_1,\dots,g_{k-1},h_k)$. Since $g_1+\dots+g_{k-1}+h_k=g_1+\dots+g_d\in S$, this is still an $S$-sequence in $G$. Furthermore, any subsequence sum of $(g_1,\dots,g_{k-1},h_k)$ is clearly a subsequence sum of $(g_1,\dots,g_d)$. Then since the original sequence does not have a zero-sum subsequence, neither does this new sequence. Thus, for every $1\leq k\leq d$, there exists an $S$-sequence in $G$ of length $k$ with no zero-sum subsequence, so $D_S(G)>d$.

    Finally, let $n=d_S(G)+1$ and consider an $S$-sequence $(g_1,\dots,g_n)$ in $G$. If this sequence has no zero-sum subsequence, then by definition, $d_S(G)\geq n$, a contradiction. Then any $S$-sequence in $G$ of length $n$ must contain a zero-sum subsequence, so $D_S(G)\leq n=d_S(G)+1$. Then $d_S(G)<D_S(G)\leq d_S(G)+1$, so $D_S(G)=d_S(G)+1$.
\end{proof}

Considering Definition \ref{small}, one might reasonably expect that we have $d_{\{g\}}(G) = D_g(G)$, the relative Davenport constant with respect to $g$ defined by Ska{\l}ba. While this is true for $g \neq 0$, we should be careful to recall that $D_0(G)=D(G)$ and $d_{\{0\}}(G)= 0$. Notably, these invariants are almost identical for singletons. Hence, we will be able to take advantage of results on $D_g(G)$ to calculate $d_{\{g\}}(G)$ and subsequently $D_S(G)$. We presently develop a few simple results about these constants which will be helpful in the following sections.

\begin{theorem}
    \label{d in subsets}
    Let $G$ be a finite abelian group and $S_1\subseteq S_2$ nonempty subsets of $G$. Then $d_{S_1}(G)\leq d_{S_2}(G)$ and $D_{S_1}(G)\leq D_{S_2}(G)$.
\end{theorem}

\begin{proof}
    If $d_{S_1}(G)=0$, then the conclusion is immediate. Otherwise, let $d=d_{S_1}(G)$ and $\{g_1,\dots, g_d\}$ be an $S_1$-sequence in $G$ with no zero-sum subsequence. Then $g_1+\dots+g_d\in S_1\subseteq S_2$, so this is also an $S_2$-sequence of length $d$ with no $0$-subsequence. Then $d_{S_1}(G)\leq d_{S_2}(G)$. From this, it immediately follows that $D_{S_1}(G)=d_{S_1}(G)+1\leq d_{S_2}(G)+1=D_{S_2}(G)$.
\end{proof}

\begin{corollary}
    \label{davenport upper bound}
    Let $G$ be a finite abelian group and $S$ a nonempty subset of $G$. Then $$D_S(G)\leq D(G).$$
\end{corollary}

\begin{theorem}
    \label{d in quotients}
    Let $G$ be a finite abelian group with subgroup $H$, and let $G'=\quot{G}{H}$. Let $S'\subseteq G'$ and $S=\{s\in G|s+H\in S'\}$ (i.e., letting $\pi+H:G\to G'$ be the natural projection modulo $H$, $S=\pi^{-1}(S')$). Then $d_{S'}(G')\leq d_S(G)$ and $D_{S'}(G')\leq D_S(G)$.
\end{theorem}

\begin{proof}
    First, if $d_{S'}(G')=0$, then the result is trivial. Otherwise, we let $d=d_{S'}(G')$ and select an $S'$-sequence $\{g_1',\dots,g_d'\}$ in $G'$ with no zero-sum subsequence. For each $1\leq i\leq d$, we select some $g_i\in G$ such that $g_i'=g_i+H$; then $(g_1+\cdots +g_d)+H=g_1'+\cdots+g_d'\in S'$. Then by definition of $S$, $g_1+\dots+g_d\in S$. Furthermore, $\{g_1,\dots,g_d\}$ cannot have a zero-sum subsequence in $G$, since this would force $\{g_1',\dots,g_d'\}$ to have a zero-sum subsequence in $G'$. Then there exists an $S$-sequence in $G$ with no zero-sum subsequence of length $d$, so $d_S(G)\geq d=d_{S'}(G')$. It immediately follows that $D_S(G)=d_S(G)+1\geq d_{S'}(G')+1=D_{S'}(G').$
\end{proof}

\begin{proposition}
    \label{automorphic image}
    Let $G$ be a finite abelian group and $S$ a nonempty subset of $G$. Also let $\sigma\in \aut(G)$. Then $d_S(G)=d_{\sigma(S)}(G)$. In particular, if $H$ is a subgroup of $G$ and $\sigma\in \aut_H(G)$ (i.e., $\sigma$ is an automorphism of $G$ such that $\sigma(H)=H$), then for any $g\in G$, $D_{g+H}(G)=D_{\sigma(G)+H}(G)$.
\end{proposition}

\begin{proof}
    First, let $\{g_1,\dots,g_n\}$ be an $S$-sequence in $G$ with no zero-sum subsequence; in particular, let $s:=g_1+\dots+g_n\in S$. Then $\{\sigma(g_1),\dots,\sigma(g_n)\}$ is a $\sigma(S)$-sequence in $G$; specifically, $\sigma(g_1)+\dots+\sigma(g_n)=\sigma(s)\in\sigma(S)$. Moreover, since $\sigma$ is an automorphism, then if there is a zero-sum subsequence, say (after possibly reordering) $\{\sigma(g_1),\dots,\sigma(g_k)\}$, then $\sigma(g_1+\dots+g_k)=0$ and thus $g_1+\dots+g_k=0$, a contradiction. Then $\{\sigma(g_1),\dots,\sigma(g_n)\}$ is a $\sigma(S)$-sequence in $G$ of length $n$ with no zero-sum subsequence, so $d_S(G)\leq d_{\sigma(S)}(G)$. The reverse inequality follows immediately from the observation that $S=\sigma^{-1}(\sigma(S))$, with $\sigma^{-1}\in \aut(G)$. Then $d_S(G)=d_{\sigma(S)}(G)$.

    The rest of the result follows from the observation that for any subgroup $H\leq G$, $\sigma\in \aut_H(G)$, and $g\in G$, we have $\sigma(g+H)=\sigma(g)+H$.
\end{proof}

\begin{corollary}
    \label{negative automorphism}
    Let $G$ be a finite abelian group, $S$ a nonempty subset of $G$, and $-S$ the set of additive inverses of elements of $S$. Then $D_S(G)=D_{-S}(G)$. In particular, for any subgroup $H$ of $G$ and coset $\alpha\in \quot{G}{H}$, $D_\alpha(G)=D_{-\alpha}(G)$.
\end{corollary}

\begin{proof}
    The result follows immediately from the previous theorem, since $\sigma:G\to G$ defined by $\sigma(g)=-g$ for all $g\in G$ is an automorphism of $G$ which fixes any subgroup $H$.
\end{proof}

For the remainder of this section, we consider the case when $G$ is cyclic and our sets are taken to be cosets of subgroups. First, consider the following result of Ska{\l}ba on the relative Davenport constant of a cyclic group. 

\begin{lemma}\cite[Theorem 2]{skalba1993numbers}\label{skalba-cyclic}
    Let $G$ be a finite cyclic group and $a\in G$, then
    \[
    D_a(G) = \begin{cases}
        |G|& \text{for } a =0\\
        |G|-|G|/|a|& \text{for } a \neq0      
    \end{cases}
    \]
    where $|a|$ denotes the order of $a$.
\end{lemma}

Notably, since $d_{\{g\}}(G)=D_g(G)$ for all $g\neq 0$ and $d_{\{0\}}=0$, this result actually demonstrates one way in which the small $S$-relative Davenport constant behaves more nicely than the relative Davenport constant: in all cases, $d_{\{g\}}(G)=\abs{G}-\abs{G}/\abs{g}$.

\begin{theorem}
    \label{d-value for cyclic}
    Let $G$ be a finite cyclic group of order $n$ and $H=\gen{m}\leq G$ (without loss of generality, we assume $m|n$). Let $\alpha\in \quot{G}{H}$ and $g\in \alpha$ such that $1\leq g\leq m$. Then $$d_\alpha(G)=d_{\{\gcd(g,m)\}}(G)=n-\gcd(g,m).$$
\end{theorem}

\begin{proof}
    Assume without loss of generality that $G=\bZ_n$; we will treat the elements of $G$ as integers. Since $d_\alpha(G)=\max\{d_{\{a\}}(G)|a\in \alpha\}$, it will suffice by Lemma \ref{skalba-cyclic} to show that $\gcd(g,m)\leq \gcd(a,n)=\frac{n}{\abs{a}}$ for every $a\in \alpha$, with equality for some $a\in \alpha$. The inequality is straightforward; letting $d=\gcd(g,m)$, we note that there must exist $k,x,y\in\bN$ such that $n=km$, $g=xd$, and $m=yd$. Then $n=km=(ky)d$, and for any $a=g+rm\in \alpha$, $a=xd+ryd=(x+ry)d$. Then $d$ is a common divisor of $n$ and $a$, so $d=\gcd(g,m)\leq \gcd(a,n)$ for every $a\in\alpha$. We simply need to show that $\gcd(g,m)=\gcd(a,n)$ for some $a\in \alpha$.

    First, assume that $n=p^r$ for some prime $p$ and $r\in\bN$. Since $m$ is a divisor of $n$, then $H=\gen{m}=\gen{p^s}$ for some $0\leq s\leq r$. We will denote $\gcd(g,n)=p^a$; since $1\leq g\leq m$, we know that $0\leq a\leq s$. Since $p^a$ is the largest power of $p$ dividing $g$ and $a\leq s\leq r$, this immediately tells us that $$\gcd(g,m)=\gcd(g,p^s)=p^a=\gcd(g,n).$$
    Then $\gcd(g,m)=\gcd(g,n)$ with $g\in g+H=\alpha$, as desired.

    Now dropping the assumption than $n$ is a prime power, we write $n=p_1^{a_1}\cdots p_k^{a_k}$ for distinct primes $p_i$ and $a_i\in\bN$, $1\leq i\leq k$. Since $m$ is a divisor of $n$, then necessarily $m=p_1^{b_1}\cdots p_k^{b_k}$ for some $0\leq b_i\leq a_i$. Decomposing our cyclic group $G=\bZ_n$, we note that
    $$G=\bZ_n\cong \bZ_{p_1^{a_1}}\oplus\cdots\oplus \bZ_{p_k^{a_k}}.$$
    Considering the image of $H$ under this isomorphism and applying the Chinese Remainder Theorem yields
    $$H \cong   \gen{p_1^{b_1}}\oplus\cdots\oplus\gen{p_k^{b_k}}\leq \bZ_{p_1^{a_1}}\oplus\cdots\oplus \bZ_{p_k^{a_k}}.$$
    Now for each $1\leq i\leq k$, we define $g_i$ such that $g_i\equiv g\modulo{p_i^{b_i}}$ and $1\leq g_i\leq p_i^{b_i}$. Using the case when $n$ was a prime power, we note that $\gcd(g_i,p_i^{b_i})=\gcd(g_i,p_i^{a_i})$ for every $1\leq i\leq k$. By the Chinese Remainder Theorem, we can now select $a\in G$ such that $a\equiv g_i\modulo{p_i^{a_i}}$ for every $1\leq i\leq k$. Since $a\equiv g_i\equiv g\modulo{p_i^{b_i}}$ for each $1\leq i\leq k$, we have that $a\equiv g\modulo{m}$, and thus $a\in g+H=\alpha$. We note the following:
    $$\gcd(g,m)=\prod_{i=1}^k\gcd(g,p_i^{b_i})=\prod_{i=1}^k\gcd(g_i,p_i^{b_i})=\prod_{i=1}^k\gcd(g_i,p_i^{a_i})=\prod_{i=1}^k\gcd(a,p_i^{a_i})=\gcd(a,n).$$
    Then $\gcd(g,m)=\gcd(a,n)$ for some $a\in \alpha$ and $\gcd(g,m)\leq \gcd(a,n)$ for every $a\in \alpha$, so
    $$d_\alpha(G)=\max\{d_a(G)|a\in\alpha\}=\max\{n-\gcd(a,n)|a\in\alpha\}=n-\gcd(g,m)=d_{\{\gcd(g,m)\}}(G).$$
\end{proof}

\begin{corollary}\label{larger relative dav for cyclic generator}
    Let $G$ be a finite cyclic group and $H\leq G$. For any $\alpha,\beta\in \quot{G}{H}$ such that $\beta\in \gen{\alpha}$, $d_\alpha(G)\geq d_\beta(G)$.
\end{corollary}

\begin{proof}
    Identifying $G$ with $\bZ_n$ and treating the elements of $G$ as integers, we let $H=\gen{m}$ for some divisor $m$ of $n$ and select $g_1\in \alpha$ and $g_2\in\beta$ such that $1\leq g_1\leq m$ and $1\leq g_2\leq m$. Since $\beta\in \gen{\alpha}$, then there exists some $a\in\alpha$ and $s\in\bZ$ such that $g_2\equiv sa\modulo{m}$. Since $\alpha=a+H=g_1+H$, we can therefore select $t\in\bZ$ such that $g_2=sg_1+tm$.
    
    Now let $d_1=\gcd(g_1,m)$ and $d_2=\gcd(g_2,m)$. Since $d_1$ is a common divisor of $g_1$ and $m$, it must also divide $g_2=sg_1+tm$. Thus, $d_1$ is a common divisor of $g_2$ and $m$, so $d_1\leq d_2$. By Theorem \ref{d-value for cyclic}, this gives:
    $$d_\alpha(G)=n-d_1\geq n-d_2=d_\beta(G).$$
\end{proof}

\begin{corollary}\label{S-relative beta minus alpha}
    Let $G$ be a finite cyclic group and $H_1<H_2\leq G$. For some $g\in G$, we denote $\alpha=g+H_1$ and $\beta=g+H_2$; in particular, note that $\alpha\subsetneq \beta$. Then $d_\beta(G)=d_{\beta\backslash\alpha}(G)$. Equivalently, there exists a $\beta$-sequence $\{g_1,\dots,g_d\}$ of length $d=d_\beta(G)$ with no $0$-subsequence in $G$ whose sum does not lie in the proper subset $\alpha$.
\end{corollary}

\begin{proof}
    Since $\beta\backslash\alpha\subset\beta$, then Proposition \ref{d in subsets} immediately tells us that $d_{\beta}(G)\geq d_{\beta\backslash\alpha}(G)$. We will show the reverse inequality holds as well.

    As usual, we identify $G$ with $\bZ_n$ for some $n\in\bN$ and $H_1$ and $H_2$ with subgroups $\gen{m_1}$ and $\gen{m_2}$ for some divisors $m_1$ and $m_2$ of $n$. Since $H_1$ is a proper subgroup of $H_2$, we note that $m_2$ must be a proper divisor of $m_1$. We factor these integers into primes:
    $$n=p_1^{a_1}\cdots p_k^{a_k};\quad m_1=p_1^{c_1}\cdots p_k^{c_k};\quad m_2=p_1^{b_1}\cdots p_k^{b_k}.$$
    Since $m_2$ must be a proper divisor of $m_1$, we may assume without loss of generality that $b_1<c_1$.

    Following the construction from Theorem \ref{d-value for cyclic}, we let $1\leq g\leq m_2$ be an element of $\beta$ and select $g_i\equiv g\modulo{p_i^{b_i}}$ with $1\leq g_i\leq p_i^{b_i}$ for each $1\leq i\leq k$. This gives $\gcd(g_i,p_i^{b_i})=\gcd(g_i,p_i^{a_i})$, so selecting $b\in G$ such that $b\equiv g_i\modulo{p_i^{a_i}}$ for every $1\leq i\leq k$ gives an element $b\in \beta$ such that $\gcd(g,m_2)=\gcd(b,n)$. If this element $b$ happens to lie outside of $\alpha$ already, then
    $$d_\beta(G)=n-\gcd(g,m_2)=n-\gcd(b,n)=d_{\{b\}}(G)\leq d_{\beta\backslash\alpha}(G).$$

    Now assume that this element $b$ we constructed in the previous step lies in $\alpha$. Since we assumed that $b_1<c_1$, we now choose $g_1'=g_1+p_1^{b_1}$ and $b'\in G$ such that $b'\equiv g_1'\modulo{p_1^{a_1}}$ and $b'\equiv g_i\modulo{p_i^{a_i}}$ for $2\leq i\leq k$. Note that $b'\equiv b\modulo{m_2}$, so $b'\in \beta$. However, note that $b'-b\equiv p_1^{b_i}\modulo{p_1^{c_1}}$, so $b'\not\equiv b\modulo{p_1^{c_1}}$. In particular, since $b\in \alpha$, $b'\notin \alpha$. However, since $\gcd(g_1',p_1^{b_1})=\gcd(g_1',p_1^{a_1})$, it follows exactly as in the proof of Theorem \ref{d-value for cyclic} that $\gcd(g,m_2)=\gcd(b',n)$. Then once again,
    $$d_\beta(G)=n-\gcd(g,m_2)=n-\gcd(b',n)=d_{\{b'\}}(G)\leq d_{\beta\backslash\alpha}(G).$$
    Thus, $d_\beta(G)=d_{\beta\backslash\alpha}(G)$.
\end{proof}

The previous two corollaries seem as if they should naturally hold for any abelian group $G$, but the more general results defy simple proof. We include them here as conjectures; we will see later that in cases where these conjectures hold, the main results of this paper may actually be stated much more simply.

\begin{conjecture}
    \label{larger d for generator}
    Let $G$ be a finite abelian group and $H\leq G$. For any $\alpha,\beta\in \quot{G}{H}$ such that $\beta\in\gen{\alpha}$, $d_\alpha(G)\geq d_\beta(G)$.
\end{conjecture}

\begin{conjecture}
    \label{d in subgroups}
    Let $G$ be a finite abelian group and $H_1<H_2\leq G$. For some $g\in G$, we denote $\alpha=g+H_1$ and $\beta=g+H_2$; in particular, note that $\alpha\subsetneq \beta$. Then $d_\beta(G)=d_{\beta\backslash\alpha}(G)$. Equivalently, there exists a $\beta$-sequence $\{g_1,\dots,g_d\}$ of length $d=d_\beta(G)$ with no zero-sum subsequence in $G$ whose sum does not lie in the proper subset $\alpha$.
\end{conjecture}

\begin{remark}
    Recall that given an order $R$ in a number field, $\Cl(\Rbar)\cong \quot{\Cl(R)}{\ker(\tau)}$, where $\tau:\Cl(R)\to\Cl(\Rbar)$ is defined by $\tau([J])=[J\Rbar]$ (where $J$ is an integral ideal of $R$ which is relatively prime to the conductor ideal $I$ of $R$). As we will see throughout this paper, we will often specifically be interested in $D_S(G)$ and $d_S(G)$ when $G=\Cl(R)$ for some order $R$ and $S=\tau^{-1}(\{[J]\})$, the set of all ideal classes in $\Cl(R)$ whose image under $\tau$ is $[J]\in\Cl(\Rbar)$ (that is, given an ideal $J$ in $\Rbar$, we let $J'\in[J]$ be an ideal relatively prime to the conductor ideal $I$ and let $S=[R\cap J']\ker(\tau)$). In this case, we will often use the notation $D_J(\Cl(R))$ and $d_J(\Cl(R))$ as shorthand for $D_S(\Cl(R))$ and $d_S(\Cl(R))$. In particular, when referring to $S=\ker(\tau)$ itself, we will often use $d_{\Rbar}(\Cl(R))$ and $D_{\Rbar}(\Cl(R))$.
\end{remark}


\section{Elasticity of Orders from the $S$-Relative Davenport Constant}
With this understanding of the $S$-relative Davenport constant of a group, we now turn our attention to a common application of the original Davenport constant: determining the elasticity of an order in a number field. Classically, it was proven in \cite{narkiewicz} that when $R$ is a ring of integers in a number field, we may use $D(\Cl(R))$, the Davenport constant of the ideal class group of $R$, to determine $\rho(R)$, the elasticity of $R$. In \cite[Theorem 3.2]{KettingerMoles2025elasticity}, this was extended to determine the elasticity of any order $R$ in a number field whose conductor ideal $I=(R:\Rbar)$ is prime in $\Rbar$. In hindsight, one will see almost immediately that this result can be stated in terms of the $S$-relative Davenport constant.

\begin{proposition}
\label{old result restated}
    Let $R$ be an order in a number field with conductor ideal $P$, where $P$ is prime as an ideal of $\Rbar$. Then:
    \begin{enumerate}
        \item if $P$ is principal in $\Rbar$, then
        $$\rho(R)=\max\left\{\frac{D(\Cl(R))}{2},\frac{D_{\Rbar}(\Cl(R))+1}{2}\right\};$$
        \item otherwise,
        $$\rho(R)=\frac{D(\Cl(R))}{2}.$$
    \end{enumerate}
\end{proposition}

In this section, we will find a natural generalization of part (1) of this result. In particular, we will determine the elasticity of any order $R$ in a number field $K$ whose conductor ideal $I$ is principal as an ideal of $\Rbar$ and prime as an ideal of $R$ (note that this covers but does not imply the case when $I$ is prime in $\Rbar$). To see how to handle this, we first produce a characterization of when $I$ is prime in $R$ that will help us partition this problem into two pieces.

\begin{proposition}
    \label{prime ideal of R}
    Let $R$ be an order in a number field with conductor ideal $I$. The following are equivalent:
    \begin{enumerate}
        \item $I$ is a prime ideal of $R$;
        \item $I=R\cap P$ for every prime ideal $P$ of $\Rbar$ containing $I$;
        \item $I=R\cap P$ for some prime ideal $P$ of $\Rbar$.
    \end{enumerate}
\end{proposition}

\begin{proof}
    Assume that $I$ is a prime ideal of $R$, and let $P$ be a prime ideal of $\Rbar$ containing $I$. Then $R\cap P$ is a prime ideal of $R$ containing $I$. Since $\dim(R)=1$ and $I\neq\{0\}$, then $R\cap P=I$. Thus, $(1)\implies (2)$. The implications $(2)\implies (3)$ and $(3)\implies (1)$ are immediate.
\end{proof}

In order to quickly handle the case when $I$ is a non-primary ideal of $\Rbar$, we defer to the following result of Halter-Koch.

\begin{proposition}
    \label{infinite elasticity}
    \cite[Corollary 4]{halter1995elasticity}
    Let $R$ be an order in an algebraic number field, $\Rbar$ its integral closure, and $I$ its conductor ideal.
    \begin{enumerate}
        \item If for some prime ideal $P$ of $R$ there is more than one prime ideal of $\Rbar$ lying over $P$, then $\rho(R)=\infty$.
        \item If for every prime ideal $P$ of $R$ there is exactly one prime ideal of $\Rbar$ lying over $P$, then $R$ has accepted (finite) elasticity.
    \end{enumerate}
    Note that by \cite[Theorem 3.2]{subringrelations}, it is sufficient to check only the prime ideals containing $I$.
\end{proposition}

\begin{corollary}
    \label{nonprimary conductor ideal}
    Let $R$ be an order in a number field with conductor ideal $I$. If $I$ is a prime ideal of $R$ such that $I$ is a non-primary ideal of $\Rbar$, then $R$ has infinite elasticity.
\end{corollary}

\begin{proof}
    Assume that $I$ is a prime ideal of $R$ which is a non-primary ideal of $\Rbar$. Then there exist distinct prime ideals $P$ and $Q$ of $\Rbar$ which contain $I$; it follows immediately from Propositions \ref{prime ideal of R} and \ref{infinite elasticity} that $R$ has infinite elasticity.
\end{proof}

We now turn our attention to the case when $I$ is a primary ideal of $\Rbar$ which is prime in $R$. By Proposition \ref{prime ideal of R}, this is equivalent to $I$ being a power of a prime ideal $P$ of $\Rbar$ with $I=R\cap P$. As mentioned previously, we will be aiming to generalize part (1) of Proposition \ref{old result restated}; that is, we will also be assuming that $I$ is a principal ideal of $\Rbar$. Before determining the elasticity of an order with such a conductor, we present lemmata that will be useful throughout this section and those that follow. The first three lemmata ensure that prime factorization of ideals in $R$ relatively prime to the conductor behaves very similarly to prime factorization of ideals in $\Rbar$.

\begin{lemma}\cite[Theorem 3.6]{conrad}
    \label{comax to I means invertible}
    Let $R$ be an order in a number field with conductor ideal $I$. If $J$ is an ideal of $R$ which is relatively prime to $I$, then $J$ is an invertible ideal of $R$.
\end{lemma}
 
\begin{lemma}\cite[Corollary 3.11]{conrad}
    \label{comax to I factors in R}
    Let $R$ be an order in a number field with conductor ideal $I$. Any ideal of $R$ which is relatively prime to $I$ has unique factorization into prime ideals of $R$ relatively prime to $I$. Moreover, all but finitely many prime ideals in $R$ are relatively prime to $I$.
\end{lemma}

\begin{lemma}\cite[Theorem 2.6]{picavet}
    \label{primes in ideal classes}
    Let $R$ be an order in a number field. Then every ideal class in $\Cl(R)$ contains infinitely many prime ideals of $R$.
\end{lemma}

The next two lemmata will help us relate ideals of $\Rbar$ relatively prime to the conductor to the ideals of $R$ relatively prime to the conductor. In particular, there is a one-to-one correspondence between these ideals, and this correspondence respects primality and prime factorization.

\begin{lemma}
    \label{go up or go down}\cite[Lemma 2.6]{KettingerMoles2025elasticity}
    Let $R$ be an order in a number field with conductor ideal $I$. If $J$ is an ideal of $R$ comaximal to $I$, then $J\Rbar\cap R=J$; if $J$ is prime, then $J\Rbar$ is prime as well. 
    
    Moreover, if $J$ is an $\Rbar$-ideal comaximal to $I$, then $(R\cap J)\Rbar=J$; if $J$ is prime, then $R\cap J$ is prime as well.
\end{lemma}

\begin{lemma}
    \label{prime factors in R}
    Let $R$ be an order in a number field with conductor ideal $I$. If $A$ is an ideal of $R$ relatively prime to $I$ with prime ideal factorization $A=Q_1\cdots Q_k$, then $A\Rbar$ is an ideal of $\Rbar$ relatively prime to $I$ with prime ideal factorization $A\Rbar=(Q_1\Rbar)\cdots(Q_k\Rbar)$.

    Moreover, if $J$ is an ideal of $\Rbar$ relatively prime to $I$ with prime ideal factorization $J=P_1\cdots P_k$, then $R\cap J$ is an ideal of $R$ relatively prime to $I$ with prime ideal factorization $R\cap J=(R\cap P_1)\cdots (R\cap P_k)$.
\end{lemma}

\begin{proof}
    Let $A$ be an ideal of $R$ relatively prime to $I$. By Lemma \ref{comax to I factors in R}, we can factor $A$ into a product of prime ideals of $R$, say $A=Q_1\dots Q_k$. Then $$A\Rbar=(Q_1\cdots Q_k)\Rbar=(Q_1\cdots Q_k)\Rbar^k=(Q_1\Rbar)\dots(Q_k\Rbar).$$
    Since Lemma \ref{go up or go down} tells us that each $Q_i\Rbar$ is a prime ideal of $\Rbar$, this is the (unique) prime factorization of $A\Rbar$. Moreover, since $A$ is relatively prime to $I$, then there exist $\alpha\in A$ and $\beta\in I$ such that $\alpha+\beta=1$. Since $\alpha\in A\Rbar$ as well, then $A\Rbar$ is also relatively prime to $I$.

    Now let $J$ be an ideal of $\Rbar$ be an ideal of $\Rbar$ relatively prime to $I$ with prime ideal factorization $J=P_1\cdots P_k$. Since $J$ is relatively prime to $I$, then there exist $\alpha\in J$ and $\beta\in I$ such that $\alpha+\beta=1$. Since $\alpha=1-\beta\in R\cap J$, then $R\cap J$ is also relatively prime to $I$. By Lemma \ref{comax to I factors in R}, we can therefore factor $R\cap J$ into a product of prime ideals of $R$, say $R\cap J=Q_1\dots Q_n$. By the previous part of this proof and Lemma \ref{go up or go down}, $$J=(R\cap J)\Rbar=(Q_1\Rbar)\cdots(Q_n\Rbar),$$ with each $Q_i\Rbar$ a prime ideal of $\Rbar$. Since prime ideal factorization in $\Rbar$ is unique, then $n=k$ and (after potential reordering), $P_i=Q_i\Rbar$ for each $1\leq i\leq k$. Thus, $$R\cap J=Q_1\cdots Q_k=(Q_1\Rbar\cap R)\cdots(Q_k\Rbar\cap R)=(R\cap P_1)\cdots (R\cap P_k).$$
\end{proof}

These final lemmata will give us ways to more easily construct irreducibles in $R$ and elements in $R$ of large elasticity. Of particular note here is Lemma \ref{prime factors of irreducible}, which demonstrates exactly how the $S$-relative Davenport constant will be useful to us.

\begin{lemma}
    \label{remove other primes}\cite[Lemma 2.7]{KettingerMoles2025elasticity}
    Let $R$ be an order in a number field $K$ with conductor ideal $I$. Let $\alpha\in R$ be a nonzero nonunit with $\rho_R(\alpha)>1$, and suppose that $\alpha=\beta\pi$ for some $\beta\in\Rbar$ and $\pi\in R$, with $\pi$ a prime element of $\Rbar$ relatively prime to $I$. Then $\beta\in R$ and $\rho_R(\beta)\geq \rho_R(\alpha)$.
\end{lemma}

\begin{lemma}
    \label{factor relatively prime to I}
    Let $R$ be an order in a number field with conductor ideal $I$. Let $\alpha\in\Rbar$, and let $\alpha\Rbar=P_1\cdots P_sQ_1\cdots Q_t$ be the (unique) factorization of $\alpha\Rbar$ into prime ideals of $\Rbar$, with each $P_i\mid I$ and $Q_j\nmid I$ for $1\leq i\leq s$ and $1\leq j\leq t$. Then $\alpha$ has a nonunit divisor lying in $R$ which is relatively prime to $I$ if and only if $\{[R\cap Q_1],\dots,[R\cap Q_t]\}$ has a zero-sum subsequence in $\Cl(R)$.
\end{lemma}

\begin{proof}
    First, assume that $\alpha$ has some nonunit divisor $r\in R$, with $r$ relatively prime to $I$. Then necessarily (after potential reordering), there exists $1\leq k\leq t$ such that $r\Rbar=Q_1\cdots Q_k$. By Lemmas \ref{go up or go down} and \ref{prime factors in R},
    $$rR=(rR)\Rbar\cap R=r\Rbar\cap R=(R\cap Q_1)\cdots(R\cap Q_k),$$
    so $\{[R\cap Q_1],\dots,[R\cap Q_k]\}$ is a zero-sum subsequence of $\{[R\cap Q_1],\dots,[R\cap Q_t]\}$ in $\Cl(R)$.

    For the converse, assume that $\{[R\cap Q_1],\dots,[R\cap Q_t]\}$ has a zero-sum subsequence in $\Cl(R)$; without loss of generality, assume that this zero-sum subsequence is of the form $\{[R\cap Q_1],\dots,[R\cap Q_\ell]\}$ for some $1\leq \ell\leq t$. Then for some $r\in R$, $rR=(R\cap Q_1)\cdots (R\cap Q_\ell)$. By applying Lemmas \ref{go up or go down} and \ref{prime factors in R}, we get
    $$r\Rbar=(rR)\Rbar=(R\cap Q_1)\Rbar\cdots(R\cap Q_\ell)\Rbar=Q_1\dots Q_\ell.$$
    Now letting $\beta\in \Rbar$ be a generator of $P_1\cdots P_sQ_{\ell+1}\cdots Q_t$, there exists a unit $u\in U(\Rbar)$ such that $\alpha=(u\beta)r$. Thus, $\alpha$ has a nonunit divisor lying in $R$ which is relatively prime to $I$ (namely, $r$).
\end{proof}
    
\begin{lemma}
    \label{elasticity relatively prime to I}
    Let $R$ be an order in a number field with conductor ideal $I$. If $\alpha\in R$ is a nonzero, nonunit relatively prime to $I$, then $\rho_R(\alpha)\leq \max\{1,\frac{D(\Cl(R))}{2}\}$. Moreover, this inequality is sharp.
\end{lemma}

\begin{proof}
    Let $\alpha\in R$ be relatively prime to $I$. If $\rho_R(\alpha)=1$, then the result follows trivially; then assume that $\rho_R(\alpha)>1$. Let $\alpha=\pi_1\cdots\pi_m=\tau_1\cdots\tau_n$ be two factorizations of $\alpha$ into irreducibles of $R$. Since each $\pi_i$ and $\tau_j$ must be relatively prime to $I$ as well, then Lemma \ref{comax to I means invertible} ensures us that each $\pi_iR$ and $\tau_jR$ (as well as each prime factor of these ideals) is an invertible ideal of $R$ (and thus represents an ideal class in $\Cl(R)$). Since we are only looking for an upper bound on the elasticity of $\alpha$, Lemma \ref{remove other primes} tells us that it will suffice to show the result in the case that each $\pi_i$ and $\tau_j$ is a non-prime irreducible.

    Let $k$ be the number of prime ideals in the prime factorization of $\alpha R$. Since each $\tau_j$ is a non-prime irreducible, then each $\tau_j$ must factor into at least two prime ideals of $R$; thus, $2n\leq k$. Since each $\pi_i$ is an irreducible element of $R$, then if $\pi_iR=Q_1\cdots Q_s$ is the prime factorization of $\pi_iR$ in $R$, then $\{[Q_1],\dots,[Q_s]\}$ must be a zero-sum sequence in $\Cl(R)$ with no proper zero-sum subsequence. Thus, $\pi_iR$ can have at most $D(\Cl(R))$ prime factors, so $k\leq m\cdot D(\Cl(R))$. Then $2n\leq m\cdot D(\Cl(R))$, so $\frac{n}{m}\leq \frac{D(\Cl(R))}{2}$. Since this inequality holds for every pair of factorizations of $\alpha$, $\rho_R(\alpha)\leq \frac{D(\Cl(R))}{2}$. Thus, $\rho_R(\alpha)\leq\max\{1,\frac{D(\Cl(R))}{2}\}$.

    To show that this inequality is sharp, note that any irreducible element of $R$ has elasticity 1. Then it will suffice to show that when $D(\Cl(R))\geq 2$, there exists an element $\alpha\in R$ relatively prime to $I$ such that $\rho_R(\alpha)=\frac{D(\Cl(R))}{2}$. To this end, let $\{[A_1],\dots,[A_d]\}$ be a zero-sum sequence in $\Cl(R)$ with no proper zero-sum subsequence of length $d=D(\Cl(R))$. By Lemma \ref{primes in ideal classes}, we can select a prime ideal $Q_i$ of $R$ relatively prime to $I$ from each ideal class $[A_i]$. Moreover, we can select a prime ideal $Q_i'$ of $R$ relatively prime to $I$ from each ideal class $[A_i]^{-1}$. Since $\{[A_1],\dots,[A_d]\}$ (and thus $\{[A_1]^{-1},\dots,[A_d]^{-1}\})$ is a zero-sum sequence with no zero-sum subsequence, this ensures that $Q_1\cdots Q_d$, $Q_1'\cdots Q_d'$, and each $Q_iQ_i'$ are principal ideals of $R$ generated by an irreducible element of $R$. Let $\beta$, $\gamma$, and $\pi_i$, respectively, be these irreducible generators. Then for some unit $u\in U(R)$, $\alpha:=\beta\gamma=(u\pi_1)\pi_2\cdots \pi_d$, so $\rho_R(\alpha)\geq \frac{D(\Cl(R))}{2}$. By the previously shown inequality and the fact that each $Q_i$ and $Q_i'$ is relatively prime to $I$, $\alpha$ is an element of $R$ relatively prime to $I$ with elasticity exactly $\frac{D(\Cl(R))}{2}$.
\end{proof}

\begin{lemma}
    \label{prime factors of irreducible}
    Let $R$ be an order in a number field whose conductor ideal $I$ is prime as an ideal of $R$ and both principal and primary as an ideal of $\Rbar$; that is, for some $\pi\in R$ and prime ideal $P$ of $\Rbar$, $I=P^a=\pi\Rbar$ and $I=R\cap P$. If $\tau\in \Irr(R)$ and $\tau\in P$, then $\tau\Rbar$ factors into at most $d_{P^j}(\Cl(R))+j$ prime ideals of $\Rbar$, where $j\in\bN$ is the largest power of $P$ such that $\tau\in P^j$. Moreover, this upper bound is achieved for each $a\leq j<2a$.
\end{lemma}

\begin{proof}
    First, note that since $R\cap P=I$, then for any $\tau\in \Irr(R)$ such that $\tau\in P$, $\tau\in I=P^a$. Also note that if $\tau\in P^{2a}=I^2=(\pi^2)$, then for some $\beta\in \Rbar$, $\tau=\pi(\pi\beta)$. Since both $\pi$ and $\pi\beta$ are nonunits in $R$, then $\tau$ cannot be irreducible. Then letting $j\in\bN$ be maximal such that $\tau\in P^j$, we have that $j=a+i$ for some $0\leq i<a$. We can now factor the ideal $\tau\Rbar$ into prime ideals, $\tau\Rbar=P^jQ_1\cdots Q_k$, where each $Q_i$ is relatively prime to $I$. Now assume that $\{[R\cap Q_1],\dots,[R\cap Q_k]\}$ has a zero-sum subsequence in $\Cl(R)$. By Lemma \ref{factor relatively prime to I}, this means that $\tau$ must have a nonunit divisor lying in $R$ which is relatively prime to $I$, and thus $\tau=\beta\gamma$ with $\gamma\in R\backslash P$ a nonunit and $\beta\in P^j\subseteq I\subseteq R$. Then $\tau$ reduces in $R$, a contradiction. Then no such $0$-subsequence can exist, so $\{[R\cap Q_1],\dots,[R\cap Q_k]\}$ is a sequence in $\Cl(R)$ with no zero-sum subsequence for which $[Q_1\cdots Q_k]=[P^j]^{-1}$. By definition of the small $S$-relative Davenport constant and Corollary \ref{negative automorphism}, $k\leq d_{P^j}(\Cl(R))$. Thus, the number of prime divisors of $\tau\Rbar$ is $k+j\leq d_{P^j}(\Cl(R))+j$.

    Now to show that this upper bound is achieved, let $a\leq j<2a$ and $\{[A_1],\dots,[A_d]\}$ be a sequence of length $d=d_{P^j}(\Cl(R))$ in $\Cl(R)$ with no zero-sum subsequence such that $\{[A_1\Rbar],\dots,[A_d\Rbar]\}$ is a $[P^j]^{-1}$-sequence in $\Cl(\Rbar)$ (such a sequence must exist by definition of the $S$-relative Davenport constant and the discussion above). By Lemma \ref{primes in ideal classes}, we can assume without loss of generality that $A_i$ is a prime ideal of $R$ relatively prime to $I$ for each $1\leq i\leq d$, and by Lemma \ref{go up or go down}, each $Q_i:=A_i\Rbar$ is a prime ideal of $\Rbar$ relatively prime to $I$. Then let $\alpha\in\Rbar$ be any generator of the (principal) ideal $P^jQ_1\cdots Q_d$. Since $j\geq a$, then $\alpha\in I\subseteq R$.

    Assume toward a contradiction that $\alpha=\beta\gamma$ for some nonunits $\beta,\gamma\in R$. Since $R\cap P=I=P^a$ and $j<2a$, then $\beta$ and $\gamma$ cannot both be elements of $P$. Without loss of generality, assume that $\beta\in P$ and $\gamma\notin P$. Then $\alpha$ has a nonunit divisor (namely, $\gamma$) lying in $R$ which is relatively prime to $I$, so Lemma \ref{factor relatively prime to I} tells us that $\{[R\cap Q_1],\dots,[R\cap Q_d]\}=\{[A_1],\dots,[A_d]\}$ must have a zero-sum subsequence in $\Cl(R)$, a contradiction. Then $\alpha\in\Irr(R)$ such that $\alpha\Rbar$ has exactly $d+j$ prime ideal factors.
\end{proof}

With these results in hand, we are now ready to present the main theorem of this section.

\begin{theorem}
    \label{elasticity of principal order}
    Let $R$ be an order in a number field whose conductor ideal $I$ is prime as an ideal of $R$ and both principal and primary as an ideal of $\Rbar$; that is, for some $\pi\in R$ and prime ideal $P$ of $\Rbar$, $I=P^a=\pi\Rbar$ and $I=R\cap P$. Then
    $$\rho(R)=\max\left\{\frac{D(\Cl(R))}{2},\frac{D_{P^i}(\Cl(R))+1}{2}+\frac{i}{a}\middle|0\leq i\leq a-1\right\}.$$
\end{theorem}

\begin{proof}
    For ease of notation, we will let $M$ refer to the maximum presented above. To show equality, we will first show that $\rho_R(\alpha)\geq M$ for some nonzero, nonunit $\alpha\in R$; then that $\rho_R(\alpha)\leq M$ for every nonzero, nonunit $\alpha\in R$. First, note that if $M=\frac{D(\Cl(R))}{2}$, then $D(\Cl(R))\geq 2$; by Lemma \ref{elasticity relatively prime to I}, we can select a nonzero, nonunit $\alpha\in R$ relatively prime to $I$ such that $\rho_R(\alpha)=\frac{D(\Cl(R))}{2}=M$. 

    Now assume for some $0\leq i\leq a-1$ that $M=\frac{D_{P^i}(\Cl(R))+1}{2}+\frac{i}{a}$. By Lemma \ref{prime factors of irreducible}, we can select $\tau\in \Irr(R)$ such that $\tau\Rbar=P^jQ_1\cdots Q_d$, where $j=a+i$, $d=d_{P^j}(\Cl(R))$, and each $Q_i$ is a prime ideal of $\Rbar$ relatively prime to $I$. Similarly, we can let $\tau'$ be any generator of the (principal) ideal $\tau'\Rbar=P^jQ_1'\cdots Q_d'$, where $[R\cap Q_i']=[R\cap Q_i]^{-1}$ for each $1\leq i\leq d$. In the same way that we showed that $\tau$ was irreducible in the proof of Lemma \ref{prime factors of irreducible}, we can conclude that $\tau'\in \Irr(R)$. Finally, let $\pi_i$ be a generator of the principal ideal $Q_iQ_i'$ for each $1\leq i\leq d$. Since $[\pi_i\Rbar\cap R]=[R\cap Q_i][R\cap Q_i']=[R]$, then we can without loss of generality assume that $\pi_i\in R$. Moreover, since each $R\cap Q_i$ is non-principal, each $\pi_i\in \Irr(R)$. Thus, for some unit $u\in U(\Rbar)$,
    $$\alpha:=(\tau\tau')^a=(u\pi)\pi^{2j-1}\pi_1^a\cdots\pi_d^a.$$
    Since the left-hand side of this expression is a product of $2a$ irreducibles in $R$ and the right-hand side of this expression is a product of $2j+ad$ irreducibles in $R$, then
    $$\rho_R(\alpha)\geq \frac{2j+ad}{2a}=\frac{d_{P^j}(\Cl(R))}{2}+\frac{a+i}{a}=\frac{D_{P^i}(\Cl(R))+1}{2}+\frac{i}{a}=M.$$ Thus, $\rho(R)\geq M$.
    
    We now need to show that $\rho_R(\alpha)\leq M$ for every nonzero, nonunit $\alpha\in R$. Let $\alpha\in R$ be a nonzero nonunit, and assume we have two factorizations of $\alpha$ into irreducibles of $R$: $\alpha=\pi_1\cdots \pi_m=\tau_1\cdots \tau_n$. If $\rho_R(\alpha)=1$, then $\rho_R(\alpha)\leq M$ trivially. Then assume that $\rho_R(\alpha)\neq 1$; without loss of generality, we may also assume that $n> m$. By Lemma \ref{remove other primes}, this allows us to assume that none of the $\pi_i$'s or $\tau_j$'s are prime elements of $\Rbar$ relatively prime to $I$. After possibly reordering, let $0\leq s\leq m$ and $0\leq t\leq n$ such that $\pi_1,\dots,\pi_s,\tau_1,\dots,\tau_t\in P$ and $\pi_{s+1},\dots \pi_m,\tau_{t+1},\dots \tau_n\notin P$. If $\tau_1\Rbar=P^jQ_1\cdots Q_k$ for some $j>a$ and prime ideals $Q_i$ relatively prime to $I$, then note that $\alpha^a=(\pi_1\cdots\pi_m)^a=\tau_1^a(\tau_2\cdots\tau_n)^a$. Since $\tau_1^a\Rbar=(P^a)^j(Q_1\cdots Q_k)^a$, then for some generator $\beta\in\Rbar$ of $(Q_1\cdots Q_k)^a$, $\tau_1^a=(\pi\beta)\pi^{j-1}$, a product of at least $j>a$ irreducibles in $R$. Then $\alpha^a$ can be written simultaneously as a product of $am$ irreducibles in $R$ and strictly more than $an$ irreducibles of $R$, so $\rho_R(\alpha^a)>\frac{n}{m}$. Since we are only interested in elements of large elasticity, it will suffice to show that $\frac{n}{m}\leq M$ when $\tau_j$ lies in $P^a\backslash P^{a+1}$ for each $1\leq j\leq t$.

    Let $k$ be the (unique) number of prime ideal factors of $\alpha\Rbar$. For $1\leq \ell\leq t$, note that $\tau_\ell\Rbar$ has exactly $a$ factors of $P$. Furthermore, for $t+1\leq \ell\leq n$, since $\tau_\ell\notin P$, then $\tau_\ell\Rbar$ has at least two prime ideal factors, and these primes must be distinct from $P$. Then the number of prime factors distinct from $P$ in the prime ideal factorization of $\alpha\Rbar$ is $k-ta\geq 2(n-t)=2n-2t$.

    For $1\leq \ell\leq s$, note that $\pi_\ell\Rbar$ must have at least $a$ factors of $P$ but no more than $2a-1$; for each $1\leq \ell\leq s$, let $1\leq i_\ell<a$ such that $\pi_\ell\Rbar$ has exactly $a+i_\ell$ factors of $P$. Thus, $ta=\sum_{\ell=1}^s(a+i_\ell)$. By Lemma \ref{prime factors of irreducible}, each $\pi_\ell$ for $1\leq \ell\leq s$ can have at most $d_\ell:=d_{P^{i_\ell}}(\Cl(R))$ prime factors distinct from $P$. Similarly, by the proof of Lemma \ref{elasticity relatively prime to I}, each $\pi_\ell$ for $s+1\leq \ell\leq m$ can have at most $D(\Cl(R))$ prime factors (distinct from $P$). Thus,
    $$2n-2t\leq k-ta\leq\sum_{j=1}^sd_j+(m-s)D(\Cl(R)).$$
    Since $ta=\sum_{\ell=1}^s(a+i_\ell)$, we have that $2t=\sum_{\ell=1}^s2+\frac{2i_\ell}{a}$. Then
    $$2n\leq 2t+\sum_{\ell=1}^sd_\ell+(m-s)D(\Cl(R))=\sum_{\ell=1}^s\left[(d_\ell+2)+\frac{2i_\ell}{a}\right]+(m-s)D(\Cl(R)).$$
    Dividing by 2, this yields
    $$n\leq \sum_{\ell=1}^s\left[\frac{D_{P^{i_\ell}}(\Cl(R))+1}{2}+\frac{i_\ell}{a}\right]+(m-s)\frac{D(\Cl(R))}{2}.$$
    Note that the right-hand side of this inequality is now a sum of $m$ terms, each of which is bounded above by $M$; then $n\leq mM$, so $\frac{n}{m}\leq M$. Then $\rho_R(\alpha)\leq M$ for every nonzero, nonunit $\alpha\in R$, so $\rho(R)\leq M$. Since we already showed the reverse inequality, we can conclude that $\rho(R)=M$.
\end{proof}

Along with Proposition \ref{nonprimary conductor ideal}, this characterizes the elasticity of any order $R$ whose conductor ideal $I$ is principal in $\Rbar$ and prime in $R$. Moreover, if $I$ is a prime ideal of $\Rbar$, $I$ will automatically be a prime ideal of $R$; then case (1) of Proposition \ref{old result restated} is exactly equivalent to Theorem \ref{elasticity of principal order} when $a=1$. 

It is worth noting that Theorem \ref{elasticity of principal order}, along with Corollary \ref{davenport upper bound}, actually gives quite a narrow range for the potential elasticities of any order of the given form.

\begin{corollary}
    Let $R$ be an order in a number field whose conductor ideal $I=P^a$ is prime as an ideal of $R$ and principal as an ideal of $\Rbar$. Then
    $$\frac{D(\Cl(R))}{2}\leq\rho(R)\leq \frac{D(\Cl(R))+1}{2}+\frac{a-1}{a}<\frac{D(\Cl(R))}{2}+\frac{3}{2}.$$
\end{corollary}

Theorem \ref{elasticity of principal order} also simplifies greatly in the case that $\Rbar$ is a PID.

\begin{corollary}
    \label{pid}
    Let $R$ be an order in a number field whose conductor ideal $I=P^a$ is prime as an ideal of $R$. Also assume that the integral closure $\Rbar$ is a PID. Then
    $$\rho(R)=\frac{D(\Cl(R))+1}{2}+\frac{a-1}{a}.$$
\end{corollary}

\begin{proof}
    Since $\Rbar$ is a PID, then $I$ is automatically a principal ideal of $\Rbar$. Furthermore, since $D_{P^i}(\Cl(R))=D_{\Cl(R)}(\Cl(R))=D(\Cl(R))$ for every $i\in\bN_0$, then by Theorem \ref{elasticity of principal order},
    $$\rho(R)=\max\left\{\frac{D(\Cl(R))}{2},\frac{D_{P^i}(\Cl(R))+1}{2}+\frac{i}{a}\middle|0\leq i\leq a-1\right\}=\frac{D(\Cl(R))+1}{2}+\frac{a-1}{a}.$$
\end{proof}

Even when $\Rbar$ is not a PID, we can still make Theorem \ref{elasticity of principal order} easier to use by noting that, of course, $\frac{i}{a}$ increases with larger values of $i$. Thus, if two different values of $i$ give identical values of $D_{P^i}(\Cl(R))$, we need only look at the term with the larger value of $i$. In particular, we can make use of Proposition \ref{automorphic image} to reduce the values of $i$ we must check. Moreover, we can use Corollary \ref{negative automorphism} to obtain the following.




\begin{corollary}
    \label{simpler assuming conjecture}
    Let $R$ be an order in a number field whose conductor ideal $I$ is prime as an ideal of $R$ and both principal and primary as an ideal of $\Rbar$; that is, for some $\pi\in R$ and prime ideal $P$ of $\Rbar$, $I=P^a=\pi\Rbar$ and $I=R\cap P$. If $D_P(\Cl(R))\geq D_{P^i}(\Cl(R))$ for every $i\in\bN$, then
    $$\rho(R)=\max\left\{\frac{D(\Cl(R))}{2},\frac{D_P(\Cl(R))+1}{2}+\frac{a-1}{a}\right\}.$$
    Note that this will hold whenever Conjecture \ref{larger d for generator} applies to $G=\Cl(R)$, $H=\ker(\tau)$, and $\alpha=[P]$. In particular, by Corollary \ref{larger relative dav for cyclic generator}, this will hold whenever $\Cl(R)$ is cyclic.
\end{corollary}

To illustrate the importance of the condition that $I$ is a prime ideal of $R$, we provide an illustrative example that demonstrates the failure of the theorem when $R\cap P\neq I$.

\begin{example}
    \label{orders with different elasticity}
    Let $K=\bQ(\sqrt{2})$, $R_1=\bZ[2\sqrt{2}]$, and $R_2=\bZ[9\sqrt{2}]$. The conductor ideals of these orders are $I_1=(R_1:\cO_K)=(2)=(\sqrt{2})^2$ and $I_2=(R_2:\cO_K)=(9)=(3)^2$, both of which are squares of prime ideals. $\Cl(\cO_K)$ is trivial (i.e., $\cO_K$ is a PID) and $R_1$ and $R_2$ are both locally associated orders (see the table at \cite{quadla}), so $\Cl(R_1)\cong \Cl(\cO_K)\cong \Cl(R_2)$. Now since $R_1\cap (\sqrt{2})=I_1$, we can use Corollary \ref{pid} to conclude that $\rho(R_1)=\frac{3}{2}$. On the other hand, note that in $R_2$, $(9+9\sqrt{2})(9-9\sqrt{2})=-3^4$, with each factor being irreducible. Then $\rho(R_2)\geq \frac{4}{2}=2>\rho(R_1)$ (and in fact, $\rho(R_2)=2$; this is shown later in Example \ref{return to quadratic example}).
\end{example}

We conclude this section with an additional example of an elasticity calculation when the ring of integers is not a PID.

\begin{example}
    \label{more complex example}
    Let $K=\bQ[\alpha]$, where $\alpha$ is a root of $f(x)=x^4-x^3+5x^2-2x+1$. Using the database at \cite{lmfdb} and SageMath, we note the following:
    \begin{itemize}
        \item $\Cl(\Rbar)\cong \bZ_2$.
        \item $3\Rbar=P^2$, where $P=(3,-1+6\alpha-\alpha^2+\alpha^3)$ is a non-principal prime ideal.
        \item Letting $I=3\Rbar$ and $R=\bZ+I$, we have that $R$ is an order in $K$ with conductor ideal $I$ and $R\cap P=I$.
        \item $\abs{U(\quot{\Rbar}{I})}=72$ and $\abs{U(\quot{R}{I})}=2$.
        \item $U(\Rbar)=\{\pm\alpha^k|k\in\bZ\}$ and $\alpha^{12}$ is the first power of $\alpha$ lying in $R$, so $\abs{\quot{U(\Rbar)}{U(R)}}=12$.
    \end{itemize}
    Therefore,
    $$\abs{\Cl(R)}=\abs{\Cl(\Rbar)}\cdot\frac{\abs{U(\quot{\Rbar}{I})}}{\abs{U(\quot{R}{I})}\cdot\abs{\quot{U(\Rbar)}{U(R)}}}=2\cdot\frac{72}{2\cdot{12}}=6.$$
    Since $\Cl(R)$ must be abelian, then $\Cl(R)\cong \bZ_6$. In this copy of $\bZ_6$, $\ker(\tau)$ corresponds to the set $\{0, 2, 4\}$; since $P$ is non-principal in $\Rbar$, then by Corollary \ref{simpler assuming conjecture}, we have:
    $$\rho(R)=\max\left\{\frac{D(\Cl(R))}{2},\frac{D_P(\Cl(R))+1}{2}+\frac{1}{2}\right\}=\max\left\{\frac{6}{2},\frac{7}{2}+\frac{1}{2}\right\}=4.$$
\end{example}

\section{Orders with $\rho(R)=\rho(\Rbar)$}
Before continuing in our efforts to find the elasticities of additional orders, we take a brief diversion to refine old results and answer old conjectures. Of particular interest in this section will be orders $R$ for which $\rho(R)=\rho(\Rbar)$. In \cite[Theorem 3.4]{radicalconductor}, it was shown that $\rho(R)\geq \rho(\Rbar)$ for any order $R$. Moreover, it was shown in \cite[Theorem 1.1]{rago} that, for $R$ to be an HFD (i.e., for $\rho(R)=\rho(\Rbar)=1$), it is necessary that $R$ is an associated order. It was conjectured in \cite[Conjecture 5.2.8]{dissertation} that $\rho(R)=\rho(\Rbar)$ would imply that $R$ is associated more generally. Though we will show later in this section that this conjecture fails in general, it turns out that $\rho(R)=\rho(\Rbar)$ does imply that $R$ is a locally associated order (see \cite{subringrelations} for full definitions and characterizations).

\begin{theorem}
    \label{equal elasticity implies la}
    Let $R$ be an order in a number field $K$. If $\rho(R)=\rho(\Rbar)$, then $R$ is a locally associated order.
\end{theorem}

\begin{proof}
    First, note that if $\rho(R)=\rho(\Rbar)=1$, then it is already known that $R$ must be an associated (and thus locally associated) order. Then assume that $\rho(R)=\rho(\Rbar)>1$. In this case, Lemma \ref{elasticity relatively prime to I} gives us:
    $$\rho(\Rbar)=\frac{D(\Cl(\Rbar))}{2}\leq \frac{D(\Cl(R))}{2}\leq \rho(R)=\rho(\Rbar).$$
    Then $D(\Cl(R))=D(\Cl(\Rbar))$; by \cite[Lemma 3.3]{radicalconductor}, it follows that $\Cl(R)\cong\Cl(\Rbar)$.
\end{proof}

As an interesting corollary, we can actually use this fact to show the converse of \cite[Theorem 3.8]{radicalconductor}.

\begin{corollary}
    \label{radical ideal converse}
    Let $R$ be an order in a number field $K$ whose conductor ideal $I$ is a radical ideal of $\Rbar$. Then $\rho(R)=\rho(\Rbar)$ if and only if $R$ is an associated order.
\end{corollary}

\begin{proof}
    The proof that $\rho(R)=\rho(\Rbar)$ when $R$ is an associated order is handled in \cite[Theorem 3.8]{radicalconductor}. We will show the converse. Since $\rho(R)=\rho(\Rbar)$, then Theorem \ref{equal elasticity implies la} tells us that $R$ must be locally associated. Furthermore, since $\rho(R)=\rho(\Rbar)<\infty$, then Proposition \ref{infinite elasticity} tells us that for any two prime ideals $P$ and $Q$ dividing $I$, $R\cap P\nsubseteq Q$. Since $I$ is radical, this is sufficient to show that $R$ is ideal-preserving. Since $R$ is both ideal-preserving and locally associated, \cite[Corollary 4.13]{subringrelations} tells us that $R$ is an associated order.
\end{proof}

We now use Theorem \ref{elasticity of principal order} to answer \cite[Conjecture 5.2.8]{dissertation} in the negative.

\begin{proposition}
    \label{conjecture counterexample}
    Let $R$ be a locally associated order in a number field $K$ with principal conductor ideal $I=P^a$, where $P=\pi\Rbar$ is principal and prime as an ideal of $\Rbar$ and $a\geq 2$. Also assume that $R\cap P=I$ and $\rho(\Rbar)\geq\frac{2a-1}{a}$. Then $\rho(R)=\rho(\Rbar)$, but $R$ is not an associated order.
\end{proposition}

\begin{proof}
    First, note that since $R\cap P=I=P^a\neq P$, then $R$ is by definition not an ideal-preserving order; then $R$ is not associated. Since $P$ is principal, then we can conclude by Corollary \ref{simpler assuming conjecture} that
    $$\rho(R)=\max\left\{\frac{D(\Cl(R))}{2},\frac{D_{\Rbar}(\Cl(R))+1}{2}+\frac{a-1}{a}\right\}.$$
    Since $R$ is locally associated, then $\Cl(R)\cong\Cl(\Rbar)$; therefore, $D(\Cl(R))=D(\Cl(\Rbar))$ and $D_{\Rbar}(\Cl(R))=1$. Then
    $$\rho(R)=\max\left\{\rho(\Rbar),\frac{2a-1}{a}\right\}=\rho(\Rbar).$$
\end{proof}

\begin{example}
    Let $R=\bZ[79\sqrt{79}]$ be the index $79$ order in the quadratic number field $\bQ[\sqrt{79}]$. In this case, $\Rbar=\bZ[\sqrt{79}]$ and $I=(79)=(\sqrt{79})^2$. Using the table at \cite{quadla}, we can verify that $R$ is locally associated but not ideal-preserving, so $R\cap (\sqrt{79})=I$. Since $\rho(\Rbar)=\frac{3}{2}$, then Proposition \ref{conjecture counterexample} tells us that $\rho(R)=\rho(\Rbar)=\frac{3}{2}$.
\end{example}

\section{Elasticity of Orders with Primary Conductor in Quadratic Fields}

The results in Section 3 certainly allow us to find the elasticity of orders that were previously unapproachable, but the restrictions on the conductor ideal (that $I$ is a prime ideal of $R$ which is principal in $\Rbar$) serve to somewhat limit their scope. However, when we restrict to the realm of orders in quadratic number fields, we can relax these conditions somewhat. Not only will every order automatically have a conductor ideal which is principal in $\Rbar$ (generated by an integer), the order will actually be uniquely determined by its conductor ideal. We will also be able to take advantage of the fact that $R\cap P$ is very predictable. Moreover, as we will see, orders in $\Rbar$ lying between $R$ and $\Rbar$ will have nice properties which we can utilize. Ultimately, this will allow us to find the elasticity of any order of index $p^k$ for prime $p$ and $k\in\bN$ in any quadratic number field $K=\bQ[\sqrt{d}]$.

As before, we begin by compiling lemmata of which we will make frequent use that apply specifically to the quadratic case. Throughout this section, $K$ will refer to a quadratic number field $\bQ[\sqrt{d}]$ for some squarefree $d\in\bZ$, and $R_n=\bZ[n\alpha]$ will refer to the unique order of index $n$ in $K$, where $\alpha=\frac{1+\sqrt{d}}{2}$ if $d\equiv 1\modulo{4}$ and $\alpha=\sqrt{d}$ otherwise. For ease of notation when working with multiple orders within the same number field, we will still use the notation $\Rbar$ to refer to the integral closure of $R_n$, i.e., $\Rbar=R_1=\bZ[\alpha]$.

\begin{lemma}
    \label{R intersect P}
    Let $R_n$ be the index $n$ order in the quadratic number field $K=\bQ[\sqrt{d}]$, $p\in\bZ$ a prime dividing $n$, and $P$ any prime ideal of $\Rbar$ lying over $p$. Then:
    \begin{enumerate}
        \item if $p$ is inert or split in $K$ and $p^k|n$, then $R_n\cap P^k=p^k\bZ+n\alpha\bZ$;
        \item if $p$ is ramified in $K$ and $p^k|n$, then $R_n\cap P^{2k-1}=R_n\cap P^{2k}=p^k\bZ+n\alpha\bZ$.
    \end{enumerate}
\end{lemma}

\begin{proof}
    Let $k\in\bN$ such that $p^k|n$, and let $I_k:=p^k\bZ+n\alpha\bZ$. Since $R_n=\bZ[n\alpha]=\bZ+n\alpha\bZ$, then $I_k\subseteq R_n$ immediately. If $p$ is either inert or split in $K$, then $I_k\subseteq p^k\Rbar\subseteq P^k$, so $I_k\subseteq R_n\cap P^k$. If $p$ is ramified, then $I_k\subseteq (p^k\Rbar)\subseteq  P^{2k}\subseteq P^{2k-1}$, so $I_k\subseteq R_n\cap P^{2k}\subseteq R_n\cap P^{2k-1}$. We must now show the reverse inclusions.

    Assume that $p$ is either inert or split in $K$, and let $\beta=a+b\alpha\in R_n\cap P^k$. Since $\beta\in R_n$, then $n|b$; since $p^k|n$ and $p^k\in P^k$, this means that $p^k|b$ as well, so $b\alpha\in P^k$. Then $a=\beta-b\alpha\in \bZ\cap P^k=p^k\bZ$, so $\beta\in I_k$. Therefore, $R_n\cap P^k\subseteq I_k$. Similarly, if $p$ is ramified, we let $\beta=a+b\alpha\in R_n\cap P^{2k-1}$. Since $\beta\in R_n$, we follow a similar logic as above to conclude that $a\in\bZ\cap P^{2k-1}=p^k\bZ$. Thus, $\beta\in I_k$, so $R_n\cap P^{2k}\subseteq R_n\cap P^{2k-1}\subseteq I_k\subseteq R_n\cap P^{2k}$.
\end{proof}

This lemma, along with Proposition \ref{infinite elasticity}, allows us to immediately narrow our focus to orders whose index has no split prime factors.

\begin{proposition}
    Let $R_n$ be the index $n$ order in the quadratic number field $K=\bQ[\sqrt{d}]$. Then $R_n$ has infinite elasticity if and only if there exists a prime $p\in\bZ$ which divides $n$ and splits in $K$.
\end{proposition}

\begin{proof}
    Assume that $p|n$ for some split prime $p$. Then there are two distinct prime ideals $P$ and $Q$ of $\Rbar$ which lie over $p$. By Lemma \ref{R intersect P}, we have $R_n\cap P=p\bZ+n\alpha\bZ=R_n\cap Q$, and thus $R_n$ has infinite elasticity by Proposition \ref{infinite elasticity}. 
    
    Otherwise, for any two distinct primes $P$ and $Q$ dividing $I=n\Rbar$, $P$ and $Q$ lie over distinct primes $p$ and $q$ of $\bZ$. Using Lemma \ref{R intersect P}, $R_n\cap P=p\bZ+n\alpha\bZ\neq q\bZ+n\alpha\bZ=R_n\cap Q$. Again using Proposition \ref{infinite elasticity}, we conclude that $R_n$ has finite elasticity.
\end{proof}

The following lemma will allow us to make use of another feature of quadratic orders: the predictability of which elements conduct others into the order.

\begin{lemma}
    \label{conducts m}
    Let $R_n$ be the index $n$ order in the quadratic number field $K=\bQ[\sqrt{d}]$, and let $n=mk$ for some $m,k\in\bN$. Then:
    \begin{enumerate}
        \item $R_n+m\Rbar=R_m$;
        \item The set of elements in $\Rbar$ conducting $m$ into $R_n$ is exactly $R_k$; that is,
        $$R_k=\{\beta\in\Rbar|m\beta\in R_n\}.$$
    \end{enumerate}
\end{lemma}

\begin{proof}
    Since $R_n=\bZ[n\alpha]=\bZ+n\alpha\bZ$ and $\Rbar=\bZ[\alpha]=\bZ+\alpha\bZ$, then note that $R_n=\bZ+n\Rbar$. Thus, $R_n+m\Rbar=\bZ+n\Rbar+m\Rbar=\bZ+m\Rbar=R_m$. This completes the proof of (1).

    For (2), we note that for any $\beta=a+b\alpha\in \Rbar$, $m\beta=ma+mb\alpha\in R_n$ if and only if $n|mb$. Since $n=mk$, this holds if and only if $k|b$, which is true if and only if $\beta\in R_k$. Thus, $\beta$ conducts $m$ into $R_n$ if and only if $\beta\in R_k$.
\end{proof}

As we will see in the main theorem of this section and the corollaries that follow, finding the elasticity of a quadratic order can be much more complicated if Conjecture \ref{d in subgroups} fails. The following lemma allows us to still handle this case.

\begin{lemma}
    \label{same units not la}
    Let $R_n$ be the index $n$ order in the quadratic number field $K=\bQ[\sqrt{d}]$. Also assume that the prime (integer) divisors of $n$ are all either ramified or inert in $K$. Finally, let $b$ be a divisor of $n$ and $a$ be a proper divisor of $b$. If $U(R_a)=U(R_b)$, then $\abs{\Cl(R_a)}< \abs{\Cl(R_b)}$. In particular, if every unit in $\Rbar$ is contained in $R_b$, then $R_b$ is not a locally associated order.
\end{lemma}

\begin{proof}
    Since $a$ is a proper divisor of $b$, then $b=ak$ for some $k\in\bN\backslash\{1\}$. Let $p$ be a prime divisor of $k$, and note that if $U(R_a)=U(R_b)$, then $U(R_b)\subseteq U(R_{ap})\subseteq U(R_a)=U(R_b)$. Moreover, $\abs{\Cl(R_a)}\leq\abs{\Cl(R_{ap})}\leq \abs{\Cl(R_b)}.$ Then it will suffice to show the result when $b=ap$ for a prime integer $p$.

    By \cite[Theorem 12.12]{neukirch}, we know that for any order $R$ in an algebraic number field with conductor ideal $I$:
    $$\frac{\abs{\Cl(R)}}{\abs{\Cl(\Rbar)}}=\frac{\abs{U(\quot{\Rbar}{I})}}{\abs{U(\quot{R}{I})}\cdot\abs{\quot{U(\Rbar)}{U(R)}}}.$$ Then:
    $$\frac{\abs{\Cl(R_b)}}{\abs{\Cl(R_a)}}=\frac{\abs{\Cl(R_b)}}{\abs{\Cl(\Rbar)}}\cdot \frac{\abs{\Cl(\Rbar)}}{\abs{\Cl(R_a)}}=\frac{\abs{U(\quot{\Rbar}{(b)})}}{\abs{U(\quot{R_b}{(b)})}\cdot\abs{\quot{U(\Rbar)}{U(R_b)}}}\cdot \frac{\abs{U(\quot{R_a}{(a)})}\cdot\abs{\quot{U(\Rbar)}{U(R_a)}}}{\abs{U(\quot{\Rbar}{(a)})}}.$$

    If we assume that $U(R_a)=U(R_b)$, then $\abs{\quot{U(\Rbar)}{U(R_a)}}=\abs{\quot{U(\Rbar)}{U(R_b)}}$. Thus, we can cancel this factor from the above equation and apply \cite[Theorem 6.2]{subringrelations} to obtain
    $$\frac{\abs{\Cl(R_b)}}{\abs{\Cl(R_a)}}=\frac{L(b,d)}{L(a,d)},$$
    where $L:\bN\times\bZ\to\bN$ is the function defined in \cite[Definition 6.1]{subringrelations}. From this definition of $L$, we can see that if $p\nmid a$, then $L(b,d)=L(ap,d)=L(a,d)\cdot L(p,d)$, with $L(p,d)=p$ if $p$ is ramified and $L(p,d)=p+1$ is $p$ is inert. If $p|a$, then we can write $a=p^rk$ for some $r,k\in\bN$; again by the definition of $L$, we have $L(b,d)=L(p^{r+1}k,d)=p\cdot L(a,d)$. In either case, note that $L(b,d)>L(a,d)$, and thus $\abs{\Cl(R_b)}>\abs{\Cl(R_a)}$.

    By letting $a=1$, we can see that if $U(\Rbar)=U(R_b)$ (i.e., every unit in $\Rbar$ is contained in $R_b$), then $\abs{\Cl(R_b)}\neq \abs{\Cl(\Rbar)}$, and thus $R_b$ is not locally associated.
\end{proof}

As in Lemma \ref{prime factors of irreducible}, we now characterize the irreducible elements of $R_n$ and determine their prime ideal factors in $\Rbar$. To keep the proofs of these results as simple as possible, we handle this in stages.

\begin{lemma}
    \label{irreducibles in Rn}
    Let $R_n$ be the index $n$ order in the quadratic number field $K=\bQ[\sqrt{d}]$. Also assume that the conductor ideal $I=n\Rbar$ is a primary ideal of $\Rbar$; that is, $n=p^a$ for some ramified or inert prime $p$ and $a\in \bN$. Let $P$ be the unique prime ideal of $\Rbar$ lying over $p$. Then an element $\pi\in P$ lies in $\Irr(R_n)$ if and only if there exist $1\leq j\leq a$ and $\alpha\in \Rbar$ with no nonunit divisors lying in $R_n$ such that $\pi=p^j\alpha$ and one of the following cases holds:
    \begin{enumerate}
        \item $j=1$ and $\alpha\in U(R_n)$;
        \item $\alpha\notin P$ and $\alpha\in R_{p^{a-j}}\backslash R_{p^{a-j+1}}$;
        \item $j=a$, $p$ is ramified in $K$, and $\alpha\in P\backslash P^2$.
    \end{enumerate}
\end{lemma}

\begin{proof}
    First, we will show that any element $\pi=p^j\alpha$ which satisfies the above conditions is necessarily an irreducible element of $R_n$; then we will show the converse. In case (1), we have that $\pi=p\alpha$ for some $\alpha\in U(R_n)$ (note that such an $\alpha$ will necessarily have no nonunit divisors). Clearly $\pi\in R_n$ as a product of two elements of $R_n$. If $p$ is inert, then $\pi$ is irreducible in $\Rbar$ and therefore irreducible in $R_n$ as well. Otherwise, $(p)=P^2$ and any nontrivial factorization $\pi=\beta\gamma$ in $R_n$ would necessarily have $\beta,\gamma\in R_n\cap P\backslash (p)$. However, Lemma \ref{R intersect P} tells us that $R_n\cap P\subseteq (p)$, so $\pi$ must be irreducible in $R_n$.

    In case (2), $\pi=p^j\alpha$ for some $1\leq j\leq a$ with $\alpha\notin P$ and $\alpha\in R_{p^{a-j}}\backslash R_{p^{a-j+1}}$. By Lemma \ref{conducts m}, any such $\pi$ must lie in $R_n$. Using Lemma \ref{R intersect P}, we know that any nontrivial factorization of $\pi$ in $R_n$ must necessarily be of the form $\pi=(p^k\beta)(p^{j-k}\gamma)$ for some $0\leq k\leq j$ and $\beta,\gamma\in\Rbar$. Assume without loss of generality that $k\leq j-k$. If $k=0$, then $\alpha=\beta\gamma$ with $\beta$ a nonunit lying in $R_n$, contradicting the fact that $\alpha$ has no nonunit divisors lying in $R_n$. Then $1\leq k\leq j-k\leq j-1$. Once again using Lemma \ref{conducts m}, we know that $\beta\in R_{p^{a-k}}\subseteq R_{p^{a-j+1}}$ and $\gamma\in R_{p^{a-(j-k)}}\subseteq R_{p^{a-j+1}}$. Therefore, $\alpha=\beta\gamma\in R_{p^{a-j+1}}$, a contradiction. Then $\pi\in \Irr(R_n)$.

    In case (3), $p$ is ramified and $\pi=p^a\alpha$ with $\alpha\in P\backslash P^2$. If $\pi$ factors nontrivially in $R_n$ as $\pi=\beta\gamma$, then exactly one of $\beta\Rbar$ or $\gamma\Rbar$ must be exactly divisible by an odd power of $P$. However, Lemma \ref{R intersect P} assures us that this is impossible if this prime power is less than $2a$. Then without loss of generality, we can assume that $\beta\in P^{2a+1}$ and $\gamma\notin P$. However, this again means that $\gamma$ is a nonunit element of $R_n$ which divides $\alpha$, a contradiction. Thus, $\pi\in \Irr(R_n)$.

    We now show the converse. Assume that $\pi\in P$ is an irreducible element of $R_n$. Since $\pi\in P$, Lemma \ref{R intersect P} assures us that $\pi=p^j\alpha$ for some $j\in \bN$ and $p\nmid \alpha$. Moreover, if $j<a$, then the same lemma tells us that $\alpha\notin P$. We also note that if $j>a$, then $\pi=p^{j-a}(p^a\alpha)$ is a factorization of $\pi$ in $R_n$, a contradiction. Then either $\pi=p^j\alpha$ for some $1\leq j\leq a$ with $\alpha\notin P$; or $p$ is ramified and $\pi=p^a\alpha$ for some $\alpha\in P\backslash P^2$. In any case, note that if $\alpha$ were to have a nonunit factor $\beta\in R_n$ (i.e., if $\alpha=\beta\gamma$ for some $\gamma\in \Rbar$), then this $\beta$ must necessarily lie outside of $P$. Thus, $\beta+I\in U(\quot{R_n}{I})$, so $(\pi+I)(\beta+I)^{-1}=p^j\gamma+I\in R_n+I$. Then $p^j\gamma\in R_n$ as well, contradicting the irreducibility of $\alpha$ in $R_n$. We conclude that $\alpha$ must have no nonunit divisors lying $R_n$.

    All that remains to show is that if $\pi=p^j\alpha$ for some $1\leq j\leq a$ with $\alpha\notin P$, then either case (1) or case (2) holds. Note that since $p^j\alpha=\pi\in R_n$, then Lemma \ref{conducts m} tells us that $\alpha\in R_{p^{a-j}}$. Then suppose that $\alpha\in R_{p^{a-j+1}}$. By the same lemma, we know that $p^{j-1}\alpha\in R_n$, and thus $\alpha=p(p^{j-1}\alpha)$. If either $p^{j-1}$ or $\alpha$ is a nonunit, this contradicts the irreducibility of $\alpha$. Then either $j=1$ and $\alpha\in U(R_n)$ (note here that $n=p^a=p^{a-j+1}$) or $\alpha\notin R_{p^{a-j+1}}$.
\end{proof}

We now define a notation that will help us more easily describe the possible number of prime ideal factors of a principal ideal of $\Rbar$ generated by an irreducible element of $R_n$.

\begin{definition}
    \label{dj}
    Let $R_n$ be the index $n$ order in the quadratic number field $K=\bQ[\sqrt{d}]$. Also assume that the conductor ideal $I=n\Rbar$ is a primary ideal of $\Rbar$; that is, that $n=p^a$ for some ramified or inert prime $p$ and $a\in\bN$. Let $P$ be the unique prime ideal of $\Rbar$ lying over $p$. For each $1\leq j\leq a$, we define the constant $d_j$ as follows:
    $$d_j=\begin{cases}d_{\ker(\tau_{a-j})}(\Cl(R_n)),&\textup{if }U(R_{p^{a-j}})\neq U(R_{p^{a-j+1}});\\d_{\ker(\tau_{a-j})\backslash\ker(\tau_{a-j+1})}(\Cl(R_n)),&\ow.\end{cases}$$
    Here, $\tau_i$ refers to the surjective function $\tau_i:\Cl(R_n)\to \Cl(R_{p^i})$ defined by $\tau_i([J])=[JR_{p^i}]$ (for proof of surjectivity, see \cite[Lemma 2.13]{radicalconductor}). Note that Lemma \ref{same units not la} guarantees that when $U(R_{p^{a-j}})=U(R_{p^{a-j+1}})$, $\ker(\tau_{a-j})\backslash\ker(\tau_{a-j+1})$ is nonempty.
\end{definition}

As we will see in the proof of the following statement, $d_j$ is defined in such a way that it will give the length of the longest sequence $\{[A_1],\dots,[A_d]\}$ in $\Cl(R)$ with no zero-sum subsequence such that $\alpha \Rbar:=(A_1\cdots A_d)\Rbar$ is principal and $\alpha$ can be chosen from $R_{p^{a-j}}\backslash R_{p^{a-j+1}}$. It is especially worth noting that in cases when Conjecture \ref{d in subgroups} holds, $d_j$ will simply be equal to $d_{\ker(\tau_{a-j})}(\Cl(R_n))$ (in particular, Corollary \ref{S-relative beta minus alpha} guarantees that this holds when $\Cl(R_n)$ is cyclic).

\begin{lemma}
    \label{prime factors of irreducibles in Rn}
    Let $R_n$ be the index $n$ order in the quadratic number field $K=\bQ[\sqrt{d}]$. Also assume that the conductor ideal $I=(n)$ is a primary ideal of $\Rbar$; that is, $n=p^a$ for some ramified or inert prime $p$ and $a\in\bN$. Let $P$ be the unique prime ideal of $\Rbar$ lying over $p$. Let $\pi\in \Irr(R_n)$, and let $p^j$ be the largest power of $p$ dividing $\pi$ in $\Rbar$. Finally, let $d_j$ be as defined in Definition \ref{dj} for each $1\leq j\leq a$. Then:
    \begin{enumerate}
        \item If $j=0$, then $\pi\Rbar$ factors into at most $D(\Cl(R_n))$ prime ideals of $\Rbar$.
        \item If $j\geq 1$ and $p$ is inert, then $\pi\Rbar$ factors into at most $j+d_j$ prime ideals of $\Rbar$.
        \item If $j\geq 1$, $p$ is ramified, and $\pi\notin P^{2a+1}$, then $\pi\Rbar$ factors into at most $2j+d_j$ prime ideals of $\Rbar$.
        \item If $p$ is ramified and $\pi\in P^{2a+1}$, then $\pi\Rbar$ factors into at most $2a+D_P(\Cl(R_n))$ prime ideals of $\Rbar$.
    \end{enumerate}
    Moreover, these upper bounds are achieved for each $1\leq j\leq a$.
\end{lemma}

\begin{proof}
    From Lemma \ref{R intersect P}, we know that if $p\nmid \pi\in\Irr(R_n)$ (i.e., if $j=0$), then $\pi\notin P$. Item (1) follows immediately as a porism to the proof of Lemma \ref{elasticity relatively prime to I}; this also shows that the upper bound given in (1) is achievable.
    
    Otherwise, we have that $\pi$ is an irreducible element of $R_n$ lying in $P$, so we can apply Lemma \ref{irreducibles in Rn} to write $\pi=p^j\alpha$ for some $1\leq j\leq a $ and $\alpha\in\Rbar$ with no nonunit divisors lying in $R_n$. If $j=1$ and $\alpha\in U(R_n)$, then $\pi\Rbar$ factors either into $1\leq j+d_j$ prime ideal factors if $p$ is inert or $2\leq 2j+d_j$ prime ideal factors if $p$ is ramified.

    Now assume that $\pi$ falls under case (3) of Lemma \ref{irreducibles in Rn}; that is, $j=a$, $p$ is ramified in $K$, and $\alpha\in P\backslash P^2$. Then $\pi\Rbar=P^{2a+1}Q_1\cdots Q_k$ for some prime ideals $Q_i$ of $\Rbar$ distinct from $P$. Since $\pi$ is irreducible in $R_n$, it must be the case that $\pi$ has no nonunit divisors lying in $R_n$ which are relatively prime to $I$. Otherwise, if such a divisor $\beta$ were to exist, we could factor $\pi=\beta\gamma$ with $\beta\in R_n$ a nonunit and $\gamma\in P^{2a+1}\subseteq R_n$ a nonunit, contradicting the irreducibility of $\pi$. By Lemma \ref{factor relatively prime to I}, the sequence $\{[R_n\cap Q_1],\dots,[R_n\cap Q_k]\}$ must have no zero-sum subsequence in $\Cl(R_n)$. Since $\alpha\Rbar=PQ_1\cdots Q_k$ is principal, then $\{[Q_1],\cdots,[Q_k]\}$ must be a $[P]^{-1}$-sequence in $\Cl(\Rbar)$, meaning that $k\leq d_P(\Cl(R_n))$. Thus, $\pi\Rbar$ factors into at most $2a+1+d_P(\Cl(R_n))=2a+D_P(\Cl(R_n))$ prime ideals of $\Rbar$, completing the proof of item (4).

    To see that the upper bound in (4) is achieved, we use the definition of $d_P(\Cl(R_n))$ to select a $[P]^{-1}$-sequence $\{[Q_1],\dots,[Q_d]\}$ of length $d=d_P(\Cl(R_n))$ in $\Cl(\Rbar)$ with each $Q_i$ a prime ideal of $\Rbar$ relatively prime to $I$ such that $\{[R_n\cap Q_1],\dots,[R_n\cap Q_d]\}$ has no zero-sum subsequence in $\Cl(R_n)$. Then let $\pi=p^k\alpha$, where $\alpha$ is a principal generator of the ideal $PQ_1\cdots Q_d$. Since $R_n\cap P\subseteq P^2$, $\alpha\in P\backslash P^2$, and $\{[R_n\cap Q_1],\dots,[R_n\cap Q_d]\}$ has no zero-sum subsequence, it must be the case that $\alpha$ has no nonunit divisors lying in $R_n$. By case (3) of Lemma \ref{irreducibles in Rn}, $\pi$ is an irreducible in $R_n$ such that $\pi\Rbar$ factors into exactly $2a+1+d=2a+D_P(\Cl(R_n))$ prime ideals of $\Rbar$.

    The only case remaining to consider is when $\pi$ falls under case (2) of Lemma \ref{irreducibles in Rn}; that is, when $\alpha\notin P$ and $\alpha\in R_{p^{a-j}}\backslash R_{p^{a-j+1}}$. In this case, we factor $\alpha\Rbar=Q_1\cdots Q_k$, with each $Q_i$ a prime ideal of $\Rbar$ distinct from $P$. Since $p^j\alpha=\pi\in R_n$, then Lemma \ref{conducts m} tells us that $\alpha\in R_{p^{a-j}}$. Then $\alpha R_{p^{a-j}}=(R_{p^{a-j}}\cap Q_1)\cdots(R_{p^{a-j}}\cap Q_k)$ is principal. Furthermore, since $\alpha$ has no nonunit divisors lying in $R_n$, Lemma \ref{factor relatively prime to I} tells us that $\{[R_n\cap Q_1],\dots,[R_n\cap Q_k]\}$ must have no zero-sum subsequence in $\Cl(R_n)$. Since this sequence must also lie in $\ker(\tau_{a-j})$, it follows that $k\leq d_{\ker(\tau_{a-j})}(\Cl(R_n))$. Moreover, assume that $U(R_{p^{a-j}})=U(R_{p^{a-j+1}})$. By Lemma \ref{same units not la}, this also means that $\abs{\Cl(R_{p^{a-j}})}<\abs{\Cl(R_{p^{a-j+1}})}$, so $\ker(\tau_{a-j})\backslash \ker(\tau_{a-j+1})$ is a nonempty subset of $\ker(\tau_{a-j})$. Suppose that $\{[R_{p^{a-j+1}}\cap Q_1],\dots,[R_{p^{a-j+1}}\cap Q_k]\}$ is a zero-sum sequence in $\Cl(R_{p^{a-j+1}})$. Then for some $\beta\in R_{p^{a-j+1}}$,
    \begin{align*}
        \alpha R_{p^{a-j}}&=(R_{p^{a-j}}\cap Q_1)\cdots(R_{p^{a-j}}\cap Q_k)\\
        &=[(R_{p^{a-j+1}}\cap Q_1)\cdots(R_{p^{a-j+1}}\cap Q_k)]R_{p^{a-j}}\\
        &=(\beta R_{p^{a-j+1}})R_{p^{a-j}}=\beta R_{p^{a-j}}.
    \end{align*}
    Then there exists some $u\in U(R_{p^{a-j}})=U(R_{p^{a-j+1}})$ such that $\alpha=\beta u$, so $\alpha\in R_{p^{a-j+1}}$, a contradiction. Then when $U(R_{p^{a-j}})=U(R_{p^{a-j+1}})$, we have the added condition that $\{[R_{p^{a-j+1}}\cap Q_1],\dots,[R_{p^{a-j+1}}\cap Q_k]\}$ cannot be a zero-sum sequence in $\Cl(R_{p^{a-j+1}})$. Thus, $k\leq d_{\ker(\tau_{a-j})\backslash\ker(\tau_{a-j+1})}(\Cl(R_n))$. In any case, we see that $k\leq d_j$, so $\pi\Rbar$ factors into at most $j+d_j$ prime ideals of $\Rbar$ if $p$ is inert and at most $2j+d_j$ prime ideals of $\Rbar$ if $p$ is ramified, completing the proofs of items (2) and (3).

    All that remains to show is that the upper bounds in items (2) and (3) are achievable for each $1\leq j\leq a$. Assume first that $U(R_{p^{a-j}})=U(R_{p^{a-j+1}})$; we note again that by Lemma \ref{same units not la}, $\ker(\tau_{a-j})\backslash \ker(\tau_{a-j+1})$ is a nonempty subset of $\ker(\tau_{a-j})$. Then by definition of the small $S$-relative Davenport constant, there exists a sequence $\{[A_1],\dots,[A_{d_j}]\}$ in $\Cl(R_n)$ of length $d_j=d_{\ker(\tau_{a-j})\backslash\ker(\tau_{a-j+1})}(\Cl(R_n))$ which has no zero-sum sequence in $\Cl(R_n)$ whose product lies in $\ker(\tau_{a-j})$ but not in $\ker(\tau_{a-j+1})$. Now for $1\leq i\leq d_j$, let $Q_i$ be a prime ideal of $\Rbar$ distinct from $P$ such that $R\cap Q_i\in [A_i]$. By definition of the $\tau_i$'s, the construction of the sequence $\{[A_1],\dots,[A_{d_j}]\}$, and Lemma \ref{factor relatively prime to I}, this means that there is a principal generator $\alpha\in R_{p^{a-j}}\backslash R_{p^{a-j+1}}$ of the ideal $\alpha\Rbar=Q_1\cdots Q_k$ which has no nonunit divisors lying in $R_n$. By Lemma \ref{irreducibles in Rn}, the element $\pi=p^j\alpha$ is an irreducible in $R_n$ such that $\pi\Rbar$ factors into exactly $j+d_j$ prime ideals of $\Rbar$ if $p$ is inert and $2j+d_j$ prime ideals of $\Rbar$ if $p$ is ramified.

    If $U(R_{p^{a-j}})\neq U(R_{p^{a-j+1}})$, then we proceed in a very similar manner. The only difference is that we now relax the condition that $\{[A_1],\dots,[A_{d_j}]\}$ must avoid being a $\ker(\tau_{a-j+1})$-sequence. This allows us to select such a sequence of length $d_j=d_{\ker(\tau_{a-j})}(\Cl(R_n))$. We again find a principal generator $\alpha\in R_{p^{a-j}}$ of the ideal $\alpha\Rbar=Q_1\cdots Q_{d_j}$ and note that $\alpha$ must have no nonunit divisors lying in $R_n$. If $\alpha\notin R_{p^{a-j+1}}$ already, then $\pi=p^j\alpha$ is an irreducible in $R_n$ such that $\pi\Rbar$ factors into exactly $j+d_j$ prime ideals of $\Rbar$ if $p$ is inert and $2j+d_j$ prime ideals of $\Rbar$ is $p$ is ramified. Otherwise, if $\alpha\in R_{p^{a-j+1}}$, then we select a unit $u\in U(R_{p^{a-j}})\backslash U(R_{p^{a-j+1}})$. If $\alpha u=r\in R_{p^{a-j+1}}$, then since $\alpha+(p^{a-j+1})\in U(\quot{R_{p^{a-j+1}}}{(p^{a-j+1}}))$, $$u+(p^{a-j+1})=(r+(p^{a-j+1}))(\alpha+(p^{a-j+1}))^{-1}\in \quot{R_{p^{a-j+1}}}{(p^{a-j+1})}.$$
    Then $u\in U(R_{p^{a-j+1}})$, a contradiction. Then $u\alpha\in R_{p^{a-j}}\backslash R_{p^{a-j+1}}$ and $\pi=p^j(u\alpha)\in\Irr(R_n)$ achieves the upper bound.
\end{proof}

We now take advantage of these observations to find the elasticity of any quadratic order with primary conductor ideal.

\begin{theorem}
    \label{elasticity of Rn}
    Let $R_n$ be the index $n$ order in the quadratic number field $K=\bQ[\sqrt{d}]$. Also assume that the conductor ideal $I=(n)$ is a primary ideal of $\Rbar$; that is, $n=p^a$ for some ramified or inert prime $p$ and $a\in\bN$. Let $P$ be the unique prime ideal of $\Rbar$ lying over $p$. Finally, let $d_j$ be as defined in Definition \ref{dj} for each $1\leq j\leq a$. If $p$ is ramified, then
    $$\rho(R_n)=\max\left\{\frac{D(\Cl(R_n))}{2},a+\frac{D_P(\Cl(R_n))}{2},j+\frac{d_j}{2}\middle|1\leq j\leq a\right\}.$$
    If $p$ is inert, then
    $$\rho(R_n)=\max\left\{\frac{D(\Cl(R_n))}{2},j+\frac{d_j}{2}\middle|1\leq j\leq a\right\}.$$
    In particular, $\rho(R_n)\geq a$.
\end{theorem}

\begin{proof}
    Let $\alpha$ be a nonzero, nonunit element of $R_n$. If $\rho_{R_n}(\alpha)=1$, then the elasticity is trivially bounded above by the given maxima. Otherwise, we can use Lemma \ref{remove other primes} to remove any prime element factors of $\alpha$ relatively prime to $P$ and produce an element with at least as large elasticity. After doing so, we factor $\alpha=\pi_1\cdots\pi_s=\tau_1\cdots\tau_t$, with each $\pi_i,\tau_j\in \Irr(R_n)$. If $p$ is ramified, then Lemma \ref{R intersect P} tells us that each $\tau_j\Rbar$ must factor into at least two prime ideal factors. Furthermore, Lemma \ref{prime factors of irreducibles in Rn} tells us that each $\pi_i\Rbar$ can have at most $M=\max\left\{D(\Cl(R)),2a+D_P(\Cl(R)),2j+d_j|1\leq j\leq a\right\}$ prime ideal factors. Then letting $r$ be the number of prime ideal factors of $\alpha\Rbar$, $s\cdot M\geq r\geq 2t$, so $\frac{t}{s}\leq \frac{M}{2}$. Since this serves as an upper bound for the elasticity of any $\alpha\in R_n$,
    $$\rho(R_n)\leq\frac{M}{2}= \max\left\{\frac{D(\Cl(R_n))}{2},a+\frac{D_P(\Cl(R_n))}{2},j+\frac{d_j}{2}\middle|1\leq j\leq a\right\}.$$

    To see that this upper bound is achievable, we note that Lemma \ref{elasticity relatively prime to I} guarantees the existence of an element $\alpha\in R_n$ with elasticity at least $\frac{D(\Cl(R_n))}{2}$. We simply need to show that elements of $R_n$ can be constructed with elasticity at least $a+\frac{D_P(\Cl(R_n))}{2}$ and at least $j+\frac{d_j}{2}$ for each $1\leq j\leq a$. By Lemma \ref{prime factors of irreducibles in Rn}, there is some $\beta\in\Irr(R_n)$ such that $\beta\Rbar=P^{2a+1}Q_1\cdots Q_d$, where $d=d_P(\Cl(R_n))$. Now for each $1\leq i\leq d$, we can select some prime ideal $Q_i'$ in $\Rbar$ such that $[R_n\cap Q_i']=[R_n\cap Q_i]^{-1}$. By Lemmas \ref{factor relatively prime to I} and \ref{irreducibles in Rn}, any generator $\gamma\in R_n$ of $\gamma\Rbar=P^{2a+1}Q_1'\cdots Q_d'$ must also be irreducible in $R_n$. Moreover, for each $1\leq i\leq d$, we can select an element $\pi_i\in\Irr(R_n)$ which generates $\pi_i\Rbar=(Q_1Q_1')$. Then for some unit $u\in U(\Rbar)$,
    $$\alpha:=(u\beta)\gamma=p^{2a+1}\pi_1\cdots\pi_d$$
    is an element of $R_n$ whose elasticity is at least $a+\frac{D_P(\Cl(R_n))}{2}$.
    
    Roughly the same argument shows that if $M=2j+d_j$ for some $1\leq j\leq a$, an element of $R_n$ can be constructed with elasticity at least $\frac{M}{2}$. This is done by letting $\beta\in \Irr(R_n)$ be a principal generator of $P^{2j}Q_1\cdots Q_{d_j}$, $\gamma\in \Irr(R_n)$ a principal generator of $P^{2j}Q_1'\cdots Q_{d_j}'$, and $\pi_i\in \Irr(R_n)$ a principal generator of $Q_iQ_i'$ for each $1\leq i\leq d_j$. The only substantial difference in this case is that if $M=2j+d_j$ for some $1\leq j\leq k$, then we must be careful about the introduction of the unit $u$. In this case, however, we note that if there is any $j<i\leq a$ for which $U(R_{p^{a-i}})\neq U(R_{p^{a-i+1}})$, then by Proposition \ref{d in subsets}, $d_j\leq d_{\ker(\tau_{a-i})}(\Cl(R_n))=d_i$. This implies that $M=2j+d_j<2i+d_i$, a contradiction, so we immediately get that $U(\Rbar)=U(R_p)=\cdots=U(R_{p^{a-j}})$; that is, every unit $u\in U(\Rbar)$ lies in $R_{p^{a-j}}$. Now when we produce the element
    $$\alpha:=(u\beta)\gamma=p^{2j}\pi_1\cdots\pi_{d_j},$$
    we note that $u\in R_{p^{a-j}}=R_n+p^{a-j}\Rbar$ and $\beta\in p^j\Rbar$. Then $u\beta\in R_n$, meaning that $\alpha$ is still an element of $R_n$ whose elasticity is at least $j+\frac{d_j}{2}$. This concludes the proof in the case that $p$ is ramified.

    Now assume that $p$ is inert. The proof that the desired elasticity is achieved by an element of $R_n$ is nearly identical to the previous case, with the notable difference being that $P$ is now a principal prime ideal generated by $p\in R_n$. All that remains to show is that this is an upper bound on the elasticity of any element of $R_n$. Let $M=\max\{D(\Cl(R_n),2j+d_j|1\leq j\leq a\}$, and consider a nonzero, nonunit element $\alpha\in R_n$ factored into irreducibles of $R_n$ in two ways: $\alpha=\pi_1\cdots\pi_s=\tau_1\cdots\tau_t$. As before, we will assume that $\rho_{R_n}(\alpha)>1$ and that none of the $\pi_i$'s or $\tau_j$'s are prime unless they generate $P$. Let $r$ be the total number of prime ideals of $\Rbar$ in the factorization of $\alpha\Rbar$; $x$ the number of factors of $P$ in this factorization; and $y$ the number of factors of other prime ideals. Immediately, $r=x+y$. Now for each $1\leq i\leq t$, note that $\tau_i\Rbar$ must have at least one prime ideal factor if $\tau_i\in P$; otherwise, $\tau_i\Rbar$ must have at least two prime ideal factors distinct from $P$. That is, $2(t-x)+x=2t-x\leq r$.

    Now considering the left-hand side of the equation, let $x_i$ be the number of factors of $P$ in the prime ideal factorization of $\pi_i\Rbar$ for each $1\leq i\leq s$, and let $y_i$ be the number of prime ideal factors distinct from $P$ in this factorization. Then $x=\sum_{i=1}^sx_i$ and $y=\sum_{i=1}^sy_i$. If $\pi_i$ does not lie in $P$, then note that $x_i=0$ and $y_i\leq D(\Cl(R_n))$. On the other hand, if $\pi_i$ is exactly divisible by $p^j$ for some $1\leq j\leq k$, then $x_i=j$ and $y_i\leq d_j$ by Lemma \ref{prime factors of irreducibles in Rn}. Then $2t-x\leq r=\sum_{i=1}^s(x_i+y_i)$, and thus
    $$\frac{t}{s}\leq \frac{\sum_{i=1}^s(x_i+y_i)+x}{2s}=\frac{1}{2}\cdot\frac{\sum_{i=1}^s(2x_i+y_i)}{s}.$$
    Since this final fraction is the average of the terms $2x_i+y_i$, which for each $1\leq i\leq s$ must be equal to either $D(\Cl(R_n))$ or $2j+d_j$ for some $1\leq j\leq k$, then this average is bounded above by $M$, the largest possible value of any individual term. Since this forms an upper bound for the elasticity of any element of $R_n$, we get
    $$\rho(R_n)\leq \frac{M}{2}=\max\left\{\frac{D(\Cl(R_n))}{2},j+\frac{r_j}{2}\middle|1\leq j\leq a\right\}.$$
    Since this upper bound is also achieved as the elasticity of an element of $R_n$, this completes the proof in the case that $p$ is inert.
\end{proof}

\begin{remark}
    One will notice that in the proof of the last statement, we observed that, by Theorem \ref{d in subsets}, for any $1\leq i\leq a$, $d_{\ker(\tau_{a-i})\backslash\ker(\tau_{a-i+1})}(\Cl(R_n))\leq d_{\ker(\tau_{a-i})}(\Cl(R_n))$. Moreover, for any $1\leq i\leq j\leq a$, $d_{\ker(\tau_{a-i})}(\Cl(R_n))\leq d_{\ker(\tau_{a-j})}(\Cl(R_n))$. Thus, if any $d_j=d_{\ker(\tau_{a-j})}(\Cl(R_n))$ (in particular, if $U(R_{p^{a-j}})\neq U(R_{p^{a-j+1}})$), then necessarily $d_j\geq d_i$ for every $i\leq j$. Then when calculating the elasticity in the previous case, one only needs to check values of $j$ with $b\leq j\leq a$, where $b$ is chosen to be minimal such that $R_{p^{a-b}}$ contains the fundamental unit of $\Rbar$. This immediately gives rise to the following corollary.
\end{remark}

\begin{corollary}
    Let $R_n$ be the index $n$ order in the quadratic number field $K=\bQ[\sqrt{d}]$. Also assume that the conductor ideal $I=(R_n:\Rbar)$ is a primary ideal of $\Rbar$; that is, that $n=p^a$ for some ramified or inert prime $p$ and $a\in\bN$. Let $P$ be the unique prime ideal of $\Rbar$ lying over $p$. Finally, assume that the fundamental unit $u\in U(\Rbar)$ is not contained in $U(R_p)$. If $p$ is ramified, then
    $$\rho(R_n)=\max\left\{\frac{D(\Cl(R_n))}{2},a+\frac{D_P(\Cl(R_n))}{2},a+\frac{d_{\Rbar}(\Cl(R_n))}{2}\right\}.$$
    If $p$ is inert, then
    $$\rho(R_n)=\max\left\{\frac{D(\Cl(R_n))}{2},a+\frac{d_{\Rbar}(\Cl(R_n))}{2}\right\}.$$
\end{corollary}

Furthermore, if Conjecture \ref{d in subgroups} holds, then note that each $d_j$ is unconditionally equal to $d_{\ker(\tau_{a-j})}(\Cl(R_n))$. In particular, $d_a\geq d_j$ for each $1\leq j\leq a$. This would lead to the following corollary, which allows for this calculation to be based only on the structure of $\Cl(R_n)$ and its quotient group $\Cl(\Rbar)$ (as opposed to a sequence of quotient groups). This also has the effect of making this calculation far easier. By Corollary \ref{S-relative beta minus alpha}, we may always use this simpler version of the calculation in the event that $\Cl(R_n)$ is a cyclic group.

\begin{corollary}
    Let $R_n$ be the index $n$ order in the quadratic number field $K=\bQ[\sqrt{d}]$ such that Conjecture \ref{d in subgroups} holds for $G=\Cl(R_n)$. Also assume that $R_n$ has a primary conductor ideal; that is, that $n=p^a$ for some ramified or inert prime $p$ and $a\in\bN$. Finally, let $P$ be the unique prime ideal of $\Rbar$ lying over $p$. If $p$ is ramified, then
    $$\rho(R_n)=\max\left\{\frac{D(\Cl(R_n))}{2},a+\frac{D_{P}(\Cl(R_n))}{2},a+\frac{d_{\Rbar}(\Cl(R_n))}{2}\right\}.$$
    If $p$ is inert, then
    $$\rho(R_n)=\max\left\{\frac{D(\Cl(R_n))}{2},a+\frac{d_{\Rbar}(\Cl(R_n))}{2}\right\}.$$
\end{corollary}

Finally, since Theorem \ref{d-value for cyclic} gives us an explicit formula for the small $S$-relative Davenport constant in the case when $\Cl(R_n)$ is cyclic, we can make this formula much more explicit in this case. To do so, it is important to keep in mind that when $p$ is ramified, $P$ either falls into the principal class of $\Cl(\Rbar)$ or the unique class of order 2.

\begin{corollary}
    \label{quadratic formula when cyclic}
    Let $R_n$ be the index $n$ order in the quadratic number field $K=\bQ[\sqrt{d}]$, and assume that $\Cl(R_n)$ is cyclic of order $h'$ with quotient group $\Cl(\Rbar)$ of order $h$. Also assume that $R_n$ has a primary conductor ideal; that is, that $n=p^a$ for some ramified or inert prime $p$ and $a\in\bN$. Finally, let $P$ be the unique prime ideal of $\Rbar$ lying over $p$. If $p$ is ramified and $P$ is a principal ideal, then
    $$\rho(R_n)=\max\left\{\frac{h'}{2},a+\frac{h'-h+1}{2}\right\}=\frac{h'}{2}+\max\left\{0,a-\frac{h-1}{2}\right\}.$$
    If $p$ is ramified and $P$ is non-principal, then
    $$\rho(R_n)=\max\left\{\frac{h'}{2},a+\frac{h'-\frac{h}{2}+1}{2}\right\}=\frac{h'}{2}+\max\left\{0,a-\frac{h-2}{4}\right\}.$$
    If $p$ is inert, then
    $$\rho(R_n)=\max\left\{\frac{h'}{2},a+\frac{h'-h}{2}\right\}=\frac{h'}{2}+\max\left\{0,a-\frac{h}{2}\right\}.$$
\end{corollary}

\begin{example}
    \label{return to quadratic example}
    In Example \ref{orders with different elasticity}, we claimed that $R_9=\bZ[9\sqrt{2}]$ in $K=\bQ[\sqrt{2}]$ had elasticity exactly 2; we now verify this. Since $\Cl(R_9)$ is trivial (and is thus a cyclic group) and $n=9=3^2$ with 3 an inert prime, we use Corollary \ref{quadratic formula when cyclic} to note that $$\rho(R_9)=\frac{1}{2}+\max\left\{0,2-\frac{1}{2}\right\}=2.$$
\end{example}

\begin{example}
    \label{more complex quadratic example}
    Let $R_{6561}$ be the index $6561=3^8$ order in the quadratic number field $K=\bQ[\sqrt{987}]$. Using SageMath, we verify that $\Cl(\Rbar)\cong\bZ_4$ and 3 is a ramified prime with $3\Rbar=P^2$ for the non-principal prime ideal $P=(3,\sqrt{987})$. We also note that, in the notation of \cite[Theorem 6.2]{subringrelations}, $L(6561,987)=6561$, and letting $u=377+12\sqrt{987}$, the fundamental unit in $\Rbar$, $u^{2187}$ is the smallest power of $u$ lying in $R_{6561}$. Thus, using \cite[Theorem 12.12]{neukirch},
    $$\abs{\Cl(R_{6561})}=4\cdot \frac{6561}{2187}=12.$$
    Since $\Cl(\Rbar)\cong \bZ_4$ is a quotient group of $\Cl(R_{6561})$, we conclude that $\Cl(R_{6561})\cong \bZ_{12}$, a cyclic group. Therefore by Corollary \ref{quadratic formula when cyclic},
    $$\rho(R_{6561})=\frac{12}{2}+\max\left\{0,8-\frac{4-2}{4}\right\}=6+8-\frac{1}{2}=\frac{27}{2}.$$
\end{example}

As a final note, we turn our attention to a pair of conjectures. The first asks whether, given a half-factorial order $R$ in a number field $K$, any order $T$ with $R\subseteq T\subseteq \Rbar$ must also be an HFD. This conjecture seems to have first appeared (albeit indirectly) in \cite{kainrath} in 2005. In 2023, \cite{rago} suggested that a counterexample should exist, though an immediate corollary to \cite[Theorem 6]{halter-koch} shows that such an example could not appear in any quadratic number field. The question was also of central interest in \cite{overringHFD}. In that paper, it was shown that the conjecture holds whenever the conductor ideal $I$ of $R$ is a radical ideal of $\Rbar$. Using our results from Section 4, we can immediately produce a stronger version of \cite[Theorem 3.5]{overringHFD}.

\begin{theorem}
    \label{elasticity equal intermediate}
    Let $R$ be an order in a number field $K$ and $T$ an order in $K$ containing $R$ (i.e., $R\subseteq T\subseteq \Rbar$). If $I=(R:\Rbar)$ is radical and $\rho(R)=\rho(\Rbar)$, then $\rho(R)=\rho(T)=\rho(\Rbar)$.
\end{theorem}

\begin{proof}
    By Corollary \ref{radical ideal converse}, $R$ must be an associated order. Since $\Rbar=R\cdot U(\Rbar)\subseteq T\cdot U(\Rbar)\subseteq \Rbar$, it immediately follows that $T$ must also be an associated order. Moreover, since $I\subseteq T$, then $(T:\Rbar)$ must be an ideal of $\Rbar$ containing (dividing) $I$. Since $I$ is radical, $(T:\Rbar)$ must be radical as well. Applying Corollary \ref{radical ideal converse} again tells us that $\rho(T)=\rho(\Rbar)$.
\end{proof}

The second, far stronger statement was conjectured in \cite{dissertation}. It asks whether, given any order $R$ in a number field $K$, any order $T$ with $R\subseteq T\subseteq \Rbar$ must satisfy $\rho(R)\geq \rho(T)\geq \rho(\Rbar)$. Although this question is likely far from being settled in full generality, it follows from Theorem \ref{elasticity of Rn} and a simple set-theoretic lemma that this conjecture must hold when $R$ is a quadratic order with primary conductor. Moreover, we can find a characterization of when the elasticity of $R$ and $T$ must be equal; namely, when the elasticity of $R$ (and of $T$) is achieved by an element relatively prime to the conductor.

\begin{lemma}
    \label{preimages}
    Let $A$, $B$, and $C$ be sets with functions $f:A\to B$ and $g:B\to C$. Let $S_1\subseteq S_2\subseteq B$ and $S\subseteq C$. The following hold:
    \begin{enumerate}
        \item $(g\circ f)^{-1}(S)=f^{-1}(g^{-1}(S))$;
        \item $f^{-1}(S_2)\backslash f^{-1}(S_1)=f^{-1}(S_2\backslash S_1)$.
    \end{enumerate}
\end{lemma}
    
\begin{proof}
    Note that by definition of the preimage under a function, $x\in (g\circ f)^{-1}(S)$ if and only if $(g\circ f)(x)=g(f(x))\in S$. This holds if and only if $f(x)\in g^{-1}(S)$, or equivalently, $x\in f^{-1}(g^{-1}(S))$. Thus, $(g\circ f)^{-1}(S)=f^{-1}(g^{-1}(S))$.

    For the second item, note that $x\in f^{-1}(S_2)\backslash f^{-1}(S_1)$ if and only if $f(x)\in S_2$ but $f(x)\notin S_1$. The fact that this is equivalent to $x\in f^{-1}(S_2\backslash S_1)$ is obvious. Then $f^{-1}(S_2)\backslash f^{-1}(S_1)=f^{-1}(S_2\backslash S_1).$
\end{proof}

\begin{theorem}\label{final}
    Let $R_{m}$ and $R_n$ be the index $m$ and $n$ orders, respectively, in the quadratic number field $K=\bQ[\sqrt{d}]$. Also assume that $R_m\subsetneq R_n$ (so $n|m$) and that $R_m$ (and therefore $R_n$) has primary conductor ideal in $\Rbar$; that is, $n=p^a$ and $m=p^b$ for some inert or ramified prime $p\in\bZ$ and $a<b$. Then $\rho(R_m)\geq \rho(R_n)$, with equality if and only if $\rho(R_m)=\rho(R_n)=\frac{D(\Cl(R_n))}{2}$ (that is, the maximum elasticity is achieved by an element relatively prime to $p$).
\end{theorem}

\begin{proof}
    When finding the elasticity of the quadratic order $R_n$ with primary conductor ideal, there are three quantities we need to consider:
    $$\frac{D(\Cl(R_n))}{2};\quad a+\frac{D_P(\Cl(R_n))}{2};\quad j+\frac{d_j}{2},\:1\leq j\leq a.$$
    If we show that the corresponding quantities for $R_m$ can only grow, the result will follow.

    First, we define surjective homomorphisms $\tau_i:\Cl(R_n)\to \Cl(R_{p^i})$ and $\sigma_j:\Cl(R_m)\to \Cl(R_{p^j})$ for $0\leq i\leq a$ and $0\leq j\leq b$ such that $\tau_i([A])=[AR_{p^i}]$ and $\sigma_j([B])=[BR_{p^j}]$. In particular, note that $\tau:=\tau_0$ maps from $\Cl(R_n)$ to $\Cl(\Rbar)$; $\sigma:=\sigma_0$ maps from $\Cl(R_m)$ to $\Cl(\Rbar)$; $\sigma_a$ maps from $\Cl(R_m)$ to $\Cl(R_n)$; and $\sigma_i=\tau_i\circ\sigma_a$ for all $0\leq i\leq a$. Since $\sigma_a:\Cl(R_m)\to \Cl(R_n)$ is a surjective homomorphism, it follows that $D(\Cl(R_m))\geq D(\Cl(R_n))$.

    For the second quantity, recall by the definition of our terminology that $D_P(\Cl(R_n))=D_A(\Cl(R_n))$, where $A=\tau^{-1}(\{[P]\})$. Similarly, $D_P(\Cl(R_m))=D_B(\Cl(R_m))$, where $B=\sigma^{-1}(\{[P]\})$. By part (1) of Lemma \ref{preimages}, this tells us that $B=\sigma_a^{-1}(A)$. Then treating $\sigma_a$ as the projection of $\Cl(R_m)$ onto $\Cl(R_n)$ and using Theorem \ref{d in quotients},
    $$D_P(\Cl(R_n))=D_A(\Cl(R_n))\leq D_B(\Cl(R_m))=D_P(\Cl(R_m)).$$
    Thus, $a+D_P(\Cl(R_n))<b+D_P(\Cl(R_m))$.

    Now let $1\leq j\leq a$, and recall the definition of $d_j$ for $R_n$:
    $$d_j=\begin{cases}d_{\ker(\tau_{a-j})}(\Cl(R_n)),&\textup{if }U(R_{p^{a-j}})\neq U(R_{p^{a-j+1}});\\d_{\ker(\tau_{a-j})\backslash\ker(\tau_{a-j+1})}(\Cl(R_n)),&\ow.\end{cases}$$
    For ease of notation, we will similarly define $d_j'$ to be the corresponding values for $R_m$ for $1\leq j\leq b$ (that is, replace $a$ and $\tau$ in definition for $d_j$ with $b$ and $\sigma$, respectively). Since $\ker(\tau_{a-j})=\tau_{a-j}^{-1}(\{[R_{p^{a-j}}]\})$ and $\ker(\sigma_{a-j})=\sigma_{a-j}^{-1}(\{[R_{p^{a-j}}]\})$, then item (1) from Lemma \ref{preimages} tells us that 
    $$\ker(\sigma_{a-j})=\sigma_{a-j}^{-1}(\{[R_{p^{a-j}}]\})=\sigma_a^{-1}(\tau_{a-j}^{-1}(\{[R_{p^{a-j}}]\})=\sigma_a^{-1}(\ker(\tau_{a-j})).$$
    Again treating $\sigma_a$ as the projection of $\Cl(R_m)$ onto $\Cl(R_n)$ and using Theorem \ref{d in quotients},
    $$d_{\ker(\tau_{a-j})}(\Cl(R_n))\leq d_{\ker(\sigma_{a-j})}(\Cl(R_m)).$$
    By the same argument and using item (2) from Lemma \ref{preimages}, we know that $$\ker(\sigma_{a-j})\backslash\ker(\sigma_{a-j+1})=\sigma_a^{-1}(\ker(\tau_{a-j}))\backslash\sigma_a^{-1}(\ker(\tau_{a-j+1}))=\sigma_a^{-1}(\ker(\tau_{a-j})\backslash\ker(\tau_{a-j+1})).$$
    Once again applying this Theorem \ref{d in quotients},
    $$d_{\ker(\tau_{a-j})\backslash\ker(\tau_{a-j+1})}(\Cl(R_n))\leq d_{\ker(\sigma_{a-j})\backslash\ker(\sigma_{a-j+1})}(\Cl(R_m)).$$
    Then for $1\leq j\leq a$, we have $d_j\leq d_{(b-a)+j}'$. Thus, for the third quantity which determines the elasticities of $R_n$ and $R_m$, we have
    $$j+\frac{d_j}{2}<(b-a)+j+\frac{d_{(b-a)+j}'}{2},\: 1\leq j\leq a.$$
    By Theorem \ref{elasticity of Rn}, it follows that $\rho(R_n)\leq \rho(R_m)$.

    Now suppose that $\rho(R_m)=\rho(R_n)$. If $\rho(R_n)\neq \frac{D(\Cl(R_n))}{2}$, then Theorem \ref{elasticity of Rn} tells us that $\rho(R_n)$ must instead be equal to either $a+\frac{D_P(\Cl(R_n))}{2}$ or $j+\frac{d_j}{2}$ for some $1\leq j\leq a$. However, our earlier arguments showed us that $b+D_P(\Cl(R_m))>a+D_P(\Cl(R_n))$ and $(b-a)+j+\frac{d_{(b-a)+j}'}{2}\geq j+\frac{d_j}{2}$. In particular, this means that $\rho(R_m)>\rho(R_n)$, contradicting our assumption. Then $\rho(R_m)=\rho(R_n)=\frac{D(\Cl(R_n))}{2}$. The converse is immediate, completing the proof.
\end{proof}

\bibliographystyle{plain}
\bibliography{bibliography}

\end{document}